\documentclass[reqno,final]{amsart}
\usepackage{natbib}  
\usepackage{fancyhdr} 
\usepackage{color} 
\usepackage{hyperref} 
\usepackage{graphicx} 
\usepackage{url}

\usepackage{amssymb}
\usepackage{url}
\usepackage{comment}
\usepackage{amsmath}
\usepackage{color}
\usepackage[usenames,dvipsnames,svgnames,table]{xcolor}
\usepackage{amsthm}
\usepackage{array}
\usepackage{setspace}
\usepackage{graphicx}
\usepackage{enumerate}
\usepackage{caption}
\usepackage{subcaption}
\usepackage{mathtools}
\usepackage{listings}
\allowdisplaybreaks
\usepackage{float}
\restylefloat{table}

\definecolor{aleacolor}{rgb}{0.16,0.59,0.78}

\hypersetup{
breaklinks,
colorlinks=true,
linkcolor=aleacolor,
urlcolor=aleacolor,
citecolor=aleacolor}

\renewcommand{\headheight}{12pt}
\renewcommand{\cite}{\citet}

\theoremstyle{plain}
\newtheorem{theorem}{Theorem}[section]                                          
\newtheorem{proposition}[theorem]{Proposition}

\theoremstyle{definition}
\newtheorem{definition}[theorem]{Definition}
\theoremstyle{remark}
\newtheorem{remark}[theorem]{Remark}

\makeatletter \@addtoreset{equation}{section} \makeatother
\renewcommand\theequation{\thesection.\arabic{equation}}
\renewcommand\thefigure{\thesection.\arabic{figure}}
\renewcommand\thetable{\thesection.\arabic{table}}

\begin{document}

\title{A two-component copula with links to insurance}

\author[Ismail \textit{et~al.}]{S. Ismail\,$^{1}$\,$^{3}$\footnote{\,$^{3}$ email: {samiha.ismail@gmail.com.}}\and G. Yu\,$^{2}$\,$^{4}$\footnote{\,$^{4}$ email: {Gao.Yu@lloyds.com.}} \and G. Reinert\,$^{1}$,\,$^{5}$\footnote{\,$^{5}$ To whom correspondence should be addressed; email: reinert@stats.ox.ac.uk.}\thanks{GR acknowledges support from EPSRC grant EP/K032402/1 and from the Oxford Martin School.} \and T. Maynard\,$^{2}$\,$^{5}$\footnote{\,$^{5}$ email: {trevor.maynard@lloyds.com.}}}
\address{$^{1}$Department of Statistics, 1 South Parks Road, Oxford OX1 3TG, UK\\
$^{2}$Exposure Management Team, Lloyd's of London, London, UK.}



\begin{abstract}
This paper presents a new copula to  model dependencies between insurance entities, by considering how insurance entities are affected by both macro and micro factors. The model used to build the copula assumes that the insurance losses of two companies or lines of business are related through a random common loss factor which is then multiplied by an individual random company factor to get the total loss amounts. The new two-component copula is not Archimedean and it extends the toolkit of copulas for the insurance industry. 
\end{abstract}

\maketitle

%
%

\section{Introduction}

There are many copulas used in the insurance industry to model dependencies between different lines of businesses within or between insurance companies. Many of these copulas are not built with insurance scenarios in mind and often their assumptions break down when modelling a full insurance distribution curve. For example, the Gaussian copula, which is one of the most commonly used copulas in the insurance industry, does not have any upper tail dependence. Thus, this copula cannot model tail correlation within lines of businesses or between insurance entities which are believed to have tail correlation, such as what would be expected between two lines of business which are heavily affected by the same catastrophic event. 

To  provide a better copula solution than what is currently in-use, this paper proposes and analyses a new copula, coined the {\it Two-component (model) copula}. This copula is based on an insurance model where a dependence structure is formulated between two insurance entities by considering the effects of both macro and micro economic factors. The  underlying model of the copula is as follows;
\begin{equation*}
X_1=\sigma_1 WY_1  \text{ and } X_2=\sigma_2 WY_2 
\end{equation*}
where $X_1$ and $X_2$ are the losses experienced by two insurance companies or lines of businesses. $W$ is the macro (common) loss factor and has an exponential distribution with parameter 1 and $Y_1$ and $Y_2$ are the micro (company specific loss) factors modelled to  have inverse gamma distributions with shape parameter $\alpha_1$ and $\alpha_2$ respectively and rate parameter 1; $\sigma_1>0$ and $\sigma_2>0$ are constants, and  $W$, $Y_1$ and $Y_2$ are assumed to be  independent random variables. From this model the Two-component copula is derived; it is given in Theorem \ref{thm:tccopula}. In this paper, as well as analysing the derivation model of the copula, simulated data is generated under the two-componend model and fitted to several different copulas to gauge how different it is to copulas already available. 

The paper is structured as follows.  Chapter 2 contains a background on copulas and provides an outline of the goodness-of-fit (GoF) tests used on the simulated data. In Chapter 3 the {\it Two-component model} is derived, analysed and goodness-of-fit tests are preformed on simulated data generated from the the derivation model of the copula.The Appendix contains a detailed algorithm for the GoF test used for the copulas, as well as the colour palette used in the plots.\\

\section{Background}\label{chap:background}
\subsection{Copulas}
A copula connects the one-dimensional marginal distributions of several random variables to the multivariate distribution function of the variables. Here we concentrate on bivariate distributions. Below is a short overview;  for reference and more details see \cite{introcopula}.

For each of the following definitions we consider functions $H: DomH\subset\overline{\mathbb{R}^2}\to RanH\subset\overline{\mathbb{R}}$ with domain  $DomH = S_1\times S_2$ where $S_1,S_2$ are nonempty; $RanH$ is the range of $H$. We let ${\bf I}=[0,1]$.

\begin{definition} 
Let $B=[x_1,y_1]\times[x_2,y_2]$ be a rectangle whose vertices are in $DomH$, then the {\bf H-volume} of $B$ is given by 
\begin{equation*}
V_H(B)=H(y_1,y_2)-H(y_1,x_2)-H(x_1,y_2)+H(x_1,x_2)
\end{equation*}
$H$ is {\bf2-increasing} if $V_H(B)\geq0$, for all rectangles $B$ whose vertices are in $DomH$.
Suppose $a_i$ is the least element of $S_i$. Then $H$ is {\bf grounded} if $H(a_1,y)=0=H(x,a_2)$, for all $ (x,y)\in DomH$.
\end{definition}

\begin{definition}
A {\bf two-dimensional copula} is a function $C: {\bf I^2}\to \overline{\mathbb{R}}$ such that $C$ is grounded and 2-increasing, and
$
	C(u,1)=u $ and $ C(1,v) = v$  for all 
 $u,v\in {\bf I}$.
\end{definition}

\cite{introcopula} states in Lemma 2.1.4 that if $H$ defined above is grounded and 2-increasing, then $H$ is non-decreasing in each argument. This lemma can be used to show that $RanC = {\bf I}$. 

A key theorem for copulas is Sklar's Theorem, which uses the notion of margins, or marginal distributions. If  $b_i$ is the greatest element of $S_i$, $i=1, 2$, then the {\bf margins} of $H$ are the functions $F$ and $G$ where
$DomF=S_1$ and $F(x)=H(x,b_2)$ for all $x\in S_1,$ whereas 
$DomG=S_2$ and $G(y)=H(b_1,y$ for all $y\in S_2.$

\begin{theorem} \label{thm:sklar}
 [Sklar's Theorem] (Thm 2.3.3 \cite{introcopula}) If $H$ is a joint (cumulative) distribution function with margins $F$ and $G$, then there exists a copula $C$ such that $\forall x,y \in \overline{\mathbb{R}}$,
\begin{equation} \label{sklar}
	H(x,y)=C(F(x),G(y)).
\end{equation}
If $F$ and $G$ are continuous, then $C$ is unique; otherwise, $C$ is uniquely determined on $RanF\times RanG$. Conversely, if $C$ is a copula and F and G are distribution functions, then $H$ defined by \eqref{sklar} is a joint distribution function with margins $F$ and $G$.
\end{theorem}

Sklar's Theorem shows that a joint distribution can be split into two parts; the respective marginal distributions of the random variables and a dependence relation, given by the copula. Thus, a copula  disentangles the dependence structure of random variables from their marginal distributions. Further, since $F(x)=x$ if $F$ is the margin of a uniform distribution, the set of copulas is the set of joint distribution functions of two $U(0,1)$ random variables evaluated on $[0,1]^2$.

An advantage copulas have over joint distribution functions is that they act predictably under strictly monotone transformations of continuous random variables, see Thm 2.4.3 and 2.4.4 in \cite{introcopula}. In particular we have that if $X$ and $Y$ are continuous random variables, with copula $C_{XY}$ and if $\alpha$ and $\beta$ be strictly decreasing transformations on $RanX$ and $RanY$, respectively, then 
\begin{equation}\label{thm:copTransformations}
C_{\alpha(X)\beta(Y)}(u,v)=u+v-1+C_{XY}(1-u,1-v).
\end{equation} 
 If $\alpha$ and $\beta$ are strictly increasing,then 
\begin{equation}\label{thm:copTransformations2}C_{\alpha(X)\beta(Y)}(u,v)=C_{XY}(u,v).
\end{equation}



Lastly, since in insurance we are concerned with dependence in extreme events (i.e. one-in-two hundred years event), we use a notion of upper tail dependence and show how it relates to copulas.

\begin{definition} \label{def:upperTail}
Let $X$ and $Y$ be continuous random variables with distributions $F$ and $G$, respectively. The upper tail dependence parameter $\lambda_U$ is the limit (if it exists) of the conditional probability that $Y$ is greater than the 100t$^{th}$ percentile of $G$ given that $X$ is greater than the 100t$^{th}$ percentile of $F$ as $t$ approaches 1, i.e.
\begin{equation}\label{equ:uppertaildependence}
\lambda_U=\lim_{t\rightarrow1^-}P[Y>G^{(-1)}(t)|X>F^{(-1)}(t)].
\end{equation}
\end{definition}

\begin{theorem}\label{thm:uppertaildependence}
(Thm 5.4.2 \cite{introcopula}) Let X, Y, F, G and $\lambda_U$ be as defined in Definition \ref{def:upperTail}, and let $C$ be the copula of $X$ and Y. If the limit of Equation (\ref{equ:uppertaildependence}) exists, then 
$$
\lambda_U=2-\lim_{t\rightarrow1^-}\frac{1-C(t,t)}{1-t}=2-\delta_C'(1^-),
$$
where $\delta_C(t)=C(t,t)$ for $t\in[0,1]$.
\end{theorem}

If $\lambda_U\in(0,1]$, then $C$ has {\it upper tail dependence}, otherwise it does not have upper tail dependence.

\subsection{Dependence measures}
The most commonly used dependence measure is Pearson's Correlation, which is not a copula based measure. Pearson's Correlation can give misleading answers if the joint distribution linking two random variables does not have an elliptical distribution and is also not defined for some heavy-tailed distributions as it requires finite variances. Here, following \cite{corDoc}, Definition 5.1, we say that a random $n-$dimensional real vector $\bf{X}$ has an elliptical distribution $E_n(\mu, \Sigma, \phi)$ with parameters $\mu \in \mathbb{R}^n$ and $\Sigma$ a nonnegative definite, symmetric $n \times n$ matrix  if the characteristic function  $\varphi$ of ${\bf{X}}-\mu$ is of the form 
\begin{equation} \label{elliptic} \varphi({\bf{t}}) = \phi({\bf{t}}^T \Sigma {\bf{t}}). \end{equation}

Since it is commonly seen that insurance data comes from a heavy-tailed distribution, in this article we will instead use Kendall's tau to measure dependence. Kendall's tau is in a class of copula-based dependence measures called concordance measures, see \cite{corDoc}, which is defined  for a random vector $(X,Y)$
\begin{equation*}
\tau(X,Y)=P[(X-\widetilde{X})(Y-\widetilde{Y})>0]-P[(X-\widetilde{X})(Y-\widetilde{Y})<0].
\end{equation*}
where $(\widetilde{X},\widetilde{Y})$ is an independent copy of $(X,Y)$.
The next theorem links Kendall's tau to Pearson's correlation, which is useful for Gaussian copulas, see Subsection \ref{subsec:gaussian}. 
 
\begin{theorem} \label{thm:tauR}(Adapted from Thm 5.4 \cite{corDoc}) Let ${\bf X} \sim E_n(\mu,\Sigma,\phi)$ (see \eqref{elliptic}) with $P(X_i=\mu_i)<1$, $P(X_j=\mu_j)<1$ and $rank(\Sigma)\geq2$. Then 
\begin{equation*}
\tau(X_i,X_j)=(1-(P(X_i=\mu_i))^2)\frac{2}{\pi}arcsin(R_{ij}),
\end{equation*}
where R is the linear correlation matrix with terms $R_{ij}:=\Sigma_{ij}/\sqrt{\Sigma_{ii}\Sigma_{jj}}$. In particular, whenever $ 0< Var(X_i), Var(X_j)<\infty$ we have $\rho(X_i,X_j)\equiv R_{ij}$, thus
\begin{equation*}
\tau(X_i,X_j)=(1-(P(X_i=\mu_i))^2)\frac{2}{\pi}arcsin(\rho(X_i,X_j)).
\end{equation*}
\end{theorem}

\subsection{Types of copulas}
In this section we discuss  the {\it Gaussian copula} and the {\it Gumbel copula} as a special case of an {\it Archimedean copula}, as well as the class of {\it Extreme-value copulas}. For reference on the Gaussian, Gumbel or Archimedean copulas see \cite{corDoc}. For reference on Extreme-value copulas see \cite{Gudendorf} and \cite{evTest}. All copulas are understood to be bivariate copulas.

\subsubsection {The Gaussian Copula}\label{subsec:gaussian}

The {\it Gaussian copula} with linear correlation matrix $R\in\mathbb{R}^2$ with $R_{12}\neq1$ is given by
\begin{equation*}
C^{Ga}_R(u,v)=\int_{-\infty}^{\Phi^{-1}(u)}\int_{-\infty}^{\Phi^{-1}(v)} \frac{1}{2\pi(1-R^2_{12})^{1/2}}\exp\left(-\frac{s^2-2R_{12}st+t^2}{2(1-R^2_{12})}\right)\mathrm{d}s\mathrm{d}t.
\end{equation*}

Gaussian copulas do not have upper tail dependence (\cite{corDoc}), which suggests that even though they are one of the most common copulas used in insurance, they are not suited to this purpose as the one-in-two hundred year events (important for regulation purposes) are modelled incorrectly. 

The linear correlation matrix is usually estimated by Pearson's correlation taken from the data, but particularly for right heavy tailed distributions this could be skewed by a few large observations.  Further, as Pearson's correlation is not invariant under strictly increasing transformations of the random variables and the copula is, we can have $C_{X_1,X_2}=C_{X_1^2,X_2^2}$, but $\rho({X_1,X_2})\neq\rho({X_1^2,X_2^2})$.
A better estimator for $R_{12}$, can be derived from Theorem \ref{thm:tauR} for elliptical distributions to be $sin(\pi\hat{\tau}(X_1,X_2)/2)$, where $\hat{\tau}(X_1,X_2)$ is the estimate of Kendall's tau estimated from the data. This estimator is more robust than $\hat{\rho}$ as Kendall's tau is invariant under monotone transformations of the data. \cite{corDoc} recommends this estimator of $R_{12}$ for both elliptical and non-elliptical distributions with elliptical copulas. Thus, in later chapters when we compare the Two Component Copula with the Gaussian copula on simulated data we will use this estimator. We note that the  estimate for $R_{12}$ is a valid Gaussian copula parameter only if $\hat{\tau}(X_1,X_2)\neq1$. So a Gaussian copula can only be fitted to $(X_1,X_2)$ by this method if $\hat{\tau}(X_1,X_2)\neq1$.

\subsubsection{The Gumbel Copula}
The family of {\it Archimedean copulas} is defined using the following two definitions:
\begin{definition} \label{pseudo-inverse}
Let $\varphi$ be a continuous, strictly decreasing function from $[0,1]$ to $[0,\infty]$ such that $\varphi(1)=0$. The pseudo-inverse of $\varphi$ is the function $\varphi^{[-1]}:[0,\infty]\rightarrow[0,1]$ given by
\begin{equation*}
\varphi^{[-1]}(t)= \left\{ 
  \begin{array}{l l}
    \varphi^{-1}(t),& \quad \text{for $0\leq t \leq\varphi(0)$,}\\
    0, & \quad \text{for $\varphi(0)\leq t \leq \infty$.}\\
  \end{array} \right.
  \end{equation*}
\end{definition}
In particular, if $\varphi(0)=\infty$, then $\varphi^{[-1]}=\varphi^{-1}$.

\begin{theorem} \label{thm:arch}
(Thm 6.1 \cite{corDoc}) Let $\varphi$ be as in Definition \ref{pseudo-inverse}, and let $\varphi^{[-1]}$ be its pseudo-inverse. Let $C$ be the function from $[0,1]^2$ to $[0,1]$ given by
\begin{equation} \label{equation:arch}
C(u,v)=\varphi^{[-1]}(\varphi(u)+\varphi(v)).
\end{equation}
Then $C$ is a copula if and only if $\varphi$ is convex.
\end{theorem}
Copulas of the form (\ref{equation:arch}) are called {\it Archimedean copulas} and $\varphi$ is called the generator of the copula. From the formula it can be seen that Archimedean copulas are symmetric ($C(u,v)=C(v,u)$  for all $u,v\in[0,1]$).

The Gumbel copula is an Archimedean copula with 
$$
C_\theta^{Gum}(u,v) = \exp(-[(-\ln u)^\theta+(-\ln v)^\theta]^{1/\theta}).
$$
Gumbel copulas have an upper tail dependence coefficient of $2-2^{1/\theta}$ (\cite{corDoc}). Using Theorem 6.5 in \cite{corDoc}, if $X$ and $Y$ are random variables with a Gumbel copula with parameter $\theta$, then $\theta=(1-\tau(X,Y))^{-1}$. This expression is a valid Gumbel parameter only if $\tau(X,Y)\geq0$ and $\tau(X,Y)\neq1$. So if $\hat{\tau}$ is Kendall's tau for $X$ and $Y$ estimated from the data, a Gumbel copula can only be fitted if $1>\hat{\tau}\geq 0$.


\subsubsection{Extreme-Value Copulas}
{\it Extreme-value copulas} occur naturally in extreme event situations, and in contrast to Gaussian or Gumbel copulas do not have to be symmetric (\cite{Gudendorf}). Comparing the Two-component copula to the class of Extreme-value copulas will compliment our tool kit. 

\begin{definition} \label{def:extremValueC} (Thm 6.2.3 \cite{Gudendorf}) A bivariate copula $C$ is an {\it Extreme-value copula} if and only if 
\begin{equation*}
C(u,v)=(uv)^{A(log(v)/log(uv))}, \quad(u,v)\in(0,1]^2\setminus\{(1,1)\},
\end{equation*}
where $A:[0,1]\rightarrow[1/2,1]$ is convex and satisfies $max(t,(1-t))\leq A(t)\leq 1$ for all $t\in[0,1]$.
\end{definition}

The upper and lower bounds of $A$ correspond to perfect independence and dependence, respectively. If $U$ and $V$ have an Extreme-value copula then the conditional probability of $U$ given $V$ is an increasing function of $U$ and vice versa for $V$ given $U$. Kendall's tau of an Extreme-value copula is non-negative and is given by 
$$
\tau=\int_0^1\frac{t(1-t)}{A(t)}dA'(t).
$$
The coefficient of upper tail dependence of an Extreme-value copula simplifies to
$$
\lambda_U=2(1-A(1/2))\in[0,1],
$$
which is a decreasing function of A(1/2).

The Gumbel copula is the only Archimedean copula that is also an Extreme-value copula, with $A(t)=((t^\theta)+(1-t)^\theta)^{1/\theta}$ (\cite{GumA}).  
 
A test specifically designed to test for Extreme-value copulas is described in \cite{evTest}. This test will be used to test the simulated data to see if an extreme value copula is appropriate. This test is different to the main goodness-of-fit method mentioned in subsection \ref{subset:gof}.

\subsection{Goodness-of-fit tests for copulas} \label{subset:gof}
This section describes a method to test the fit of a copula based on \cite{Gof} and \cite{Berg}.

\cite{Gof} assesses the robustness of three goodness-of-fit tests for copulas which are based on the empirical copula process, Kendall's dependence function and the Rosenblatt's transform, respectively. His findings do not specifically show that one test was better than the other. In this study we focus on the goodness-of-fit test based on the empirical copula process because it is the most intuitive test out of the three. 
This test is based on comparing the best parametric copula under the null hypothesis with Deheuvels' empirical copula, which is defined as follows.

\begin{definition}
Let ${\bf U}=(U_1,U_2)^T$ be a vector of any two uniform random variables. Let $(u_{1i},u_{2i})^T$ for $i=1, ...., n$ be an i.i.d. sample of {\bf U} of size n. Then Deheuvels' bivariate empirical copula for {\bf U} is defined as 
\begin{equation}\label{equ:empCopula} 
C_n(v_1,v_2)\equiv\frac{1}{n}\sum_{i=1}^n{\bf1}_{(u_{1i}\leq v_1,u_{2i}\leq v_2)}, \quad v_1,v_2\in[0,1].
\end{equation}
\end{definition}

The empirical copula is similar to the well-known empirical cumulative distribution function (c.d.f.) and converges uniformly to the true underlying copula (\cite{Gof}), making it a (discontinuous) approximation of the true copula. 

To describe the goodness-of-fit test, suppose we have a random vector ${\bf X}=(X_1,X_2)^T$ containing two random variables, and suppose we have $n$ i.i.d. samples of this vector, ${\bf x_i}=(x_{1i},x_{2i})^T$  for $i=1,...,n$. in order to  avoid problems on the $[0,1]^2$ boundary we define a transformed sample as
\begin{equation}\label{eq:transformedsample} 
{\bf u_i}=(u_{1i},u_{2i})^T=\left(\frac{n}{n+1}\hat{F}_1(x_{1i}),\frac{n}{n+1}\hat{F}_2(x_{2i})\right) \text{ for } i=1,...,n, 
\end {equation}
where $
\hat{F}_j(v)=\frac{1}{n}\sum_{i=1}^n{\bf1}_{(x_{ji}\leq v)} \text{ for }j=1,2\text{ and }v\in[0,1]
$ is the empirical one-dimensional c.d.f. at $v$. 
Then, the fit of a parametric copula is assessed using a Cram\'er-von-Mises statistic;
$
\rho_{CvM}\equiv\int_{[0,1]^2}n(C_n({\bf v})-C_{\hat{\theta}}({\bf v}))^2\mathrm{d}{\bf v},
$
where $C_n$ is Deheuvels's bivariate empirical copula and $C_{\hat{\theta}}$ is the best fitting parametric copula from the parametric copula family that contains the true copula under $H_0$. The parameter of this copula ($\hat{\theta}$) is estimated using the transformed sample ($(u_{1i},u_{2i})^T$). In this study the test statistic is approximated empirically by
\begin{equation}\label{equ:CvM}
\hat{\rho}_{CvM}\equiv\sum_{i=1}^n(C_n((u_{1i},u_{2i}))-C_{\hat{\theta}}((u_{1i},u_{2i})))^2.
\end{equation}
As the distribution of this test statistic is unknown, the $p$-values of the goodness-of-fit test are approximated using a bootstrap method that can be found in Section 3.10 of \cite{Berg}; see Appendix \ref{appendix:bootstrap}.
This test performed well in the power study conducted in \cite{Berg}, where the power of nine goodness-of-fit tests for copulas were compared. 
Note that the test is independent of the assumption on the marginal distributions.

\subsection{The distribution of large insurance losses}
This  subsection explains the properties generally attributed to and a distribution used to describe large insurance losses that help to derive the model in Chapter \ref{chap:copulaMethod}.

\subsubsection{The heavy-tailed property of  large insurance}
For the five biggest insurance losses from $2001-2011$,  the range of the loss figures is  \$57.5 billion, approximately $80\%$ of the largest loss figure, which is \$72.3 billion (Hurricane Katrina).
 Further, the second largest loss, \$35.0 billion (Tohoku earthquake and tsunami)  is less than $50\%$ of the largest loss, according to  \url{http://www.businessinsider.com/the-11-most-expensive-insurance-losses-in-recent-history-2012-2}.
This is a property of right-heavy tailed distributions. There are several definitions for a heavy-tail distribution (see Theorem 2.6 in \cite{pap:heavytail}); in this paper we use the following definition:
\begin{definition} \label{def:heavytail} 
(Adapted from Thm 2.6 and Def 2.4 \cite{pap:heavytail}) The distribution function F is a (right) heavy-tailed distribution if and only if 
\begin{equation*}
\limsup_{x\to\infty}e^{\lambda x}P(X>x)=\infty \quad \forall \lambda >0.
\end{equation*}
\end{definition}

Thus, a distribution is heavy-tailed if extreme right-tail events are more likely to occur in the distribution relative to any exponential distribution.

\subsubsection{The Generalised Pareto Distribution}
The Generalised Pareto Distribution (GPD) is commonly used   to model large insurance losses; for reference see \cite{Pareto}. 
 
\begin{definition}(\cite{embrechts}) A random variable $X$ has a Generalised Pareto Distribution with location parameter $\mu\in\mathbb{R}$, scale parameter $\sigma>0$ and shape parameter $\xi\in\mathbb{R}$ (denoted by $X\sim GPD(\xi,\mu,\sigma)$) if
\begin{spacing}{1.2}
\begin{equation*}
F(x)= \left\{ 
  \begin{array}{l l}
    1-(1+ \frac{\xi(x-\mu)}{\sigma})^{-1/\xi}& \quad \text{for $\xi\neq0$}\\
    1-exp(-\frac{x-\mu}{\sigma}) & \quad \text{for $\xi=0$}\\
  \end{array} \right.
  \end{equation*}
  \end{spacing}
\noindent for $x\geq\mu$ when $\xi\geq0$, and $\mu\leq x \leq \mu-\sigma/\xi$ when $\xi<0$. In particular a random variable $X$ has a Type II Pareto distribution  with location parameter $\mu\in\mathbb{R}$, scale parameter $\sigma>0$ and shape parameter $\alpha>0$ (denoted by $X\sim P(II)(\mu,\sigma,\alpha)$) if its c.d.f. is
$$
F(x)=  1-\left(1+ \frac{(x-\mu)}{\sigma}\right)^{-\alpha}, \quad x \ge \mu.
$$
\end{definition}

Depending on  $\xi$, the GPD is related to one of three  distributions.
\begin{enumerate}
\item
If $\xi>0$ then  $GPD(\xi,\mu,\sigma) \sim P(II)(\mu,\frac{\sigma}{\xi},\frac{1}{\xi})$; 
\item if  $\xi=0$ then $GPD(\xi,\mu,\sigma) -\mu \sim Exp(\frac{1}{\sigma})$;
\item if $\xi<0$ then  $GPD(\xi,\mu,\sigma) -\mu$ is a scaled beta distribution.
\end{enumerate} 
 Since the Exponential and Beta distributions are not  heavy-tailed distributions, the rest of this section focuses on the case $\xi>0$.

Comparing the survival distribution of the Type II Pareto distribution with $e^{\lambda x}x^{-\alpha}$ for any $\lambda, \alpha>0$, we see that it is a heavy-tailed distribution.

A construction of Pareto distributions from other distributions, is a Feller-Pareto distribution, given in Theorem \ref{thm:FP}.

\begin{theorem} \label{thm:FP}
(\cite{Pareto}) Let $\mu\in\mathbb{R}$ and $\sigma, \gamma, \delta_1, \delta_2>0$. Let $U_1\sim\Gamma(\delta_1,1)$ and $U_2\sim\Gamma(\delta_2,1)$ be two independent Gamma distributions. Then 
$$
W=\mu+\sigma\left(\frac{U_1}{U_2}\right)^{\gamma}
$$
has a Feller-Pareto distribution, denoted by $W\sim FP(\mu,\sigma,\gamma,\delta_1,\delta_2)$.
Further, $P(II)(\mu, \sigma, \alpha) \sim FP(\mu, \sigma, 1, 1, \alpha)$. 
\end{theorem}

\section{The Two-component copula} \label{chap:copulaMethod}
In this section we hypothesise how insurance losses (losses) are dependent and derive and analyse a new copula ({\it Two-component model copula}) that models these hypotheses. The copula is derived by; first building a model ({\it Two-component model}) of an insurance scenario from the hypotheses, then applying Sklar's theorem to find the copula of this model. Lastly, we see how well our GoF tests perform on data generated from the Two-component model.

\subsection{The two-component model} \label{sec:simpleModel}
Preliminary to hypothesising about the dependence structure, we make the following assumptions about the marginal distributions of large insurance losses which are based on well-accepted beliefs.   
\begin{enumerate}
	\item The marginal distributions are GPDs. This assumption is recommended in \cite{embrechts} as it is an Extreme-Value theory distribution.
	\item The GPDs have $\xi\in(0,1]$, hence are Type II Pareto distributions. 
	This assumption arises as it is a  common belief that losses are heavy tailed. 
	\item The GPDs have $\mu =0$. This assumption is plausible as it translates to the assumption that 
	 no profit can be made from an insurance payout.
\end{enumerate}

The hypotheses of how losses are dependent are derived by breaking down the problem for why they would occur. Losses occurs if two conditions hold; firstly, a loss event occurred, and secondly, the loss event was underwritten by the company. For simplicity we assume these are the only two factors affecting a payout (other factors like the possibility of default are ignored). The size of the payout should be proportional to both the size of the event, which should not depend on the company because they occur on a macro level, and the level of business underwritten. So suppose that the insurance losses of two companies or lines of business, 1 and 2, in any given year are represented by the random variables $X_1$ and $X_2$ respectively. Let $W$ be a random variable representing the size of aggregate loss events in a given year, and as the amount of business written which can be affected by loss events differs between syndicates, define two more variables $Y_1$ and $Y_2$ which represent the amount of affected business underwritten in the two loss functions, 1 and 2, respectively. Then we assume that 
$X_1\propto W \mbox{ and } X_1 \propto Y_1 $ as well as $  X_2 \propto W \mbox{ and } X_2 \propto Y_2 .$

Lastly, as a companies write business before loss events happen, we assume $W$ is independent of $Y_1$ and $Y_2$, and further for simplicity we also assume $Y_1$ is independent of $Y_2$.

To construct $(X_i, W, Y_i)$ for $i=1,2$ such that; $X_i\sim P(II)(0,\sigma_i,\alpha_i)$ and is proportional to $W$ and $Y_i$, which are independent, we use the Feller-Pareto construction (Theorem \ref{thm:FP}). Since the shape parameter ($\alpha_i$) differs between loss functions we take $W\sim U_1$ (in the Theorem) and $Y_1 \sim \frac{1}{U_2}$. Our Two-component model is summarised as follows;

{\center \bf Two-component model summary\\}
{\it  Let $W\sim Exp(1)$ represent the size of the loss events that occurs in a given year and $Y_i$, where $(Y_i)^{-1}\sim\Gamma(\alpha_i,1)$, represent the level of underwritten business that can affect the loss function $i$ in a given year, for $i=1,2$. Suppose that $W$, $Y_1$ and $Y_2$ are independent. Define
\begin{eqnarray}
&X_1=\sigma_1WY_1,\label{eq:X1}\\
&X_2=\sigma_2WY_2,\label{eq:X2}
\end{eqnarray} 
\noindent where $\sigma_1,\sigma_2>0$. Then, $X_i\sim P(II)(0, \sigma_i, \alpha_i)$ and models  the loss functions $i$, for $i=1,2$.
}

The assumption that $Y_i$ has an inverse-gamma distribution is plausible as it leads to an arc shaped hazard function ($h(t)$) with limits 0 (as $t\rightarrow0^+$ and $t\rightarrow\infty$) \cite{hazardfunction}, as used in survival analysis and some mixture models \cite{Glen}. Arc shaped hazard functions can be justified in this context as the total amount of business available for underwriting is a limited resource. For small $t$, there is plenty of business for underwriting, so its easy for an insurance company to underwrite more, hence $h(t)$ increases. For large $t$, due to competition, it is difficult to find new business to underwrite so $h(t)$ decreases. 
The assumption that $W$ is exponential is made partly for convenience, but it is plausible to assume that the loss sizes follow a memoryless distribution.

\subsection{The Two-component model copula derivation}
Now we derive the {\it Two-component (model) copula} and some of its properties.

\begin{theorem}\label{thm:tccopula}
The copula for the Two-component model is
\[
 C(u_1,u_2) =
  \begin{dcases}
   u_1+u_2-1+\int^\infty_0F_{G_1}\left(\frac{w}{\left(\left(1-u_1\right)^{-\frac{1}{\alpha_1}}-1\right)}\right)&\\
   \quad\times F_{G_2}\left(\frac{w}{\left(\left(1-u_2\right)^{-\frac{1}{\alpha_2}}-1\right)}\right)e^{-w}dw &\text{if } (u_1,u_2)\in(0,1)^2\\
 u_1 & \text{if } u_1\in[0,1], u_2=1\\
  u_2 & \text{if } u_1=1, u_2\in[0,1)\\
0 & \text{otherwise},
  \end{dcases}
\]
where $G_1\sim\Gamma(\alpha_1,1)$ and $G_2\sim\Gamma(\alpha_2,1)$.
\end{theorem}

\begin{proof}
Let $X_1$ and $X_2$ be defined by equations \ref{eq:X1} and \ref{eq:X2} and let $H(x_1,x_2)$ be their joint distribution function, then by conditioning on $W$ we have for $x_1,x_2\in(0,\infty)$
\begin{equation*}
\begin{aligned}
H(x_1,x_2)
&=\Pr(X_1\leq x_1, X_2\leq x_2)\\
&=\int^\infty_0\Pr(X_1\leq x_1,X_2\leq x_2|W=w)e^{-w}dw\\
&=\int^\infty_0\left(1-F_{Y_1^{-1}}\left(\frac{w\sigma_1}{x_1}\right)\right)\left(1-F_{Y_2^{-1}}\left(\frac{w\sigma_2}{x_2}\right)\right)e^{-w}dw
\end{aligned} 
\end{equation*}
where the last line follows from the independence of $W$, $Y_1 $ and $Y_2$. 
It is straightforward to calculate that for $x\in(0,\infty)$
$$
F_{X_i}(x_i)=1-\int^\infty_0F_{Y_i^{-1}}\left(\frac{w\sigma_i}{x_i}\right)e^{-w}dw = 1 - \left(1+\frac{x_i}{\sigma_i}\right)^{-\alpha_i} $$
and hence
$$ x_i=\sigma_i\left(\left(1-F_{X_i}(x_i)\right)^{-1/\alpha_i}-1\right).
$$
Thus,
\begin{equation*}
\begin{aligned}
&H(x_1,x_2)=F_{X_1}(x_1)+F_{X_2}(x_2)-1\\
&+\int^\infty_0F_{Y_1^{-1}}\left(\frac{w}{\left(\left(1-F_{X_1}(x_1)\right)^{-\frac{1}{\alpha_1}}-1\right)}\right)F_{Y_2^{-1}}\left(\frac{w}{\left(\left(1-F_{X_2}(x_2)\right)^{-\frac{1}{\alpha_2}}-1\right)}\right)e^{-w}dw.
\end{aligned}
\end{equation*}
For $x_1=\infty$ or $x_2=\infty$, we have $F_{X_i}(\infty)=1$ for $i=1,2$ and
\begin{equation*}
H(\infty,x_2)=P(X_1\leq\infty,X_2\leq x_2)=P(X_2\leq x_2)=F_{X_2}(x_2),\quad  \forall x_2\in\overline{\mathbb{R}}
\end{equation*}
$$H(x_1,\infty)=F_{X_1}(x_1), \quad \forall x_1\in\overline{\mathbb{R}}.$$
If $x_1\leq0$ and $x_2<\infty$, 
\begin{equation*}
F_{X_1}(x_1)=0\text{ and }H(x_1,x_2)=P(X_1\leq x_1,X_2\leq x_2)=0,\quad \forall x_2\in\overline{\mathbb{R}}.
\end{equation*}
Similarly, if $x_1<\infty$ and $x_2\leq0$, then 
$F_{X_2}(x_2)=0\text{ and }H(x_1,x_2)=0$ for all $x_1\in\overline{\mathbb{R}}.$
With $C(u_1,u_2)$ as  in the statement of the theorem, and with the parameters $\alpha_1$ and $\alpha_2$ from the marginal distributions of $X_1$ and $X_2$ respectively,
\begin{equation}\label{disCop}
H(x_1,x_2)=C(F_{X_1}(x_1),F_{X_2}(x_2))\text{ for all }(x_1,x_2)\in\overline{\mathbb{R}}^2.
\end{equation}
By Sklar's theorem (Theorem \ref{thm:sklar}), we see that as $F_{X_i}$ is continuous with $RanF_{X_i}=[0,1]$, $C$ is the uniquely determined function on $[0,1]^2$ such that Equation \eqref{disCop} holds, the function $C$ is a copula and further, $C$ is the copula of the Two-component model random variables defined in  (\ref{eq:X1}) and (\ref{eq:X2}).
\end{proof}

Before we discuss the properties of this copula there are three points to mention. Firstly, the only method we know to estimate $\alpha_1$ and $\alpha_2$ is to assume the dataset has Pareto Type II marginal distributions, and to fit these margins. Hence, the use of this copula is limited to when this assumption holds and further, this method increases the error in the GoF test for the copula. Secondly, the copula requires $\alpha_1, \alpha_2>0$. Lastly, even though the Two-component model variables, $X_1$ and $X_2$, are dependent on the parameters $\sigma_1$ and $\sigma_2$, these parameters do not feature in the copula, hence these parameters do not need to be estimated when fitting the copula. \\

\subsection{Properties of the two-component copula}
To limit the notation we first look at the copula $C_{-X_1,-X_2}$, which we denote by the function $Z$, and use this function to find the behaviour of the Two-component copula using the equations below, which were inferred from  \eqref{thm:copTransformations}. Let
\begin{equation}
\begin{aligned}
Z(u,v)
:&=C_{-X_1,-X_2}(u,v)=u+v-1+C_{X_1,X_2}(1-u,1-v)\\
\text{and }z(u,v)&={\partial^2\over\partial u\partial v}Z(u,v)= c_{X_1,X_2}(1-u,1-v),\label{eq:w}
\end{aligned}
\end{equation}
where $c_{X_1,X_2}(u,v):={\partial^2\over\partial u\partial v}C_{-X_1,-X_2}(u,v)$. So for $(u,v)\in(0,1)^2$
\begin{equation*}
Z(u,v)=\int^\infty_0F_{G_1}\left(\frac{w}{\left(u^{-\frac{1}{\alpha_1}}-1\right)}\right)F_{G_2}\left(\frac{w}{\left(v^{-\frac{1}{\alpha_2}}-1\right)}\right)e^{-w}dw.
\end{equation*}
Direct verification shows that
\begin{equation}  \label{thm:densityZ}
\begin{aligned}
z(u,v)
&=\frac{u^{-(1/\alpha_1+1)}v^{-(1/\alpha_2+1)}(u^{-1/\alpha_1}-1)^{\alpha_2}(v^{-1/\alpha_2}-1)^{\alpha_1}}{(\alpha_1+\alpha_2+1)B(\alpha_1+1,\alpha_2+1)(u^{-1/\alpha_1}v^{-1/\alpha_2}-1)^{(\alpha_1+\alpha_2+1)}},
\end{aligned}
\end{equation}
where $f_{G_i}$ is the probability density function (p.d.f.) of $G_i\sim\Gamma(\alpha_i,1)$ and $B(\cdot,\cdot)$ is the Beta function.
Using (\ref{thm:densityZ}) and (\ref{eq:w}) it is straight-forward to calculate that if $c_{X_1,X_2}$ is as defined in \ref{eq:w}, then $c_{X_1,X_2}:(0,1)^2\rightarrow\mathbb{R}$ such that  
\begin{equation}\label{eq:tcdensity}
\begin{aligned}
c&_{X_1,X_2}(u,v)=\\
&\frac{(1-u)^{-(\frac{1}{\alpha_1}+1)}(1-v)^{-(\frac{1}{\alpha_2}+1)}((1-u)^{-\frac{1}{\alpha_1}}-1)^{\alpha_2}((1-v)^{-\frac{1}{\alpha_2}}-1)^{\alpha_1}}{(\alpha_1+\alpha_2+1)B(\alpha_1+1,\alpha_2+1)((1-u)^{-\frac{1}{\alpha_1}}(1-v)^{-\frac{1}{\alpha_2}}-1)^{(\alpha_1+\alpha_2+1)}},
\end{aligned}
\end{equation}
where $f_{G_i}$ is the probability density function of $G_i\sim\Gamma(\alpha_i,1)$ and $B(\cdot,\cdot)$ is the Beta function.
Now using Sklar's theorem, on $(0,1)^2$ we know that $c_{X_1,X_2}$ is equal to the joint density function of two uniform random variables which have copula $C_{X_1,X_2}$.


\begin{remark}
Let $C_{X_1,X_2}$ and $c_{X_1,X_2}$ be as defined in Theorem \ref{thm:tccopula} and (\ref{eq:tcdensity}). If $\alpha_1=1$ and $\alpha_2=1$ then it is straighforward to verify that 
$$
C_{X_1,X_2}(u,v)=\frac{uv}{u+v-uv},
 \mbox{ and } 
c_{X_1,X_2}(u,v)=\frac{2uv}{(u+v-uv)^3}.
$$
\end{remark}

 Figures \ref{fig:twoCompDensity1} and \ref{fig:twoCompDensity2} show  $c_{X_1,X_2}$ for five different values of $(\alpha_1,\alpha_2)$ namely; $(0.5,0.7), (1,1), (1,2)$, $(30,35)$ and $(1,35)$. There are two plots for each pair; one plot showing the whole $c_{X_1,X_2}$ graph and the other just showing the part of the $c_{X_1,X_2}$ graph which falls in the unit cube. The plots illustrate  that the 
 more similar $\alpha_1$ and $\alpha_2$ are, the more symmetric the copula; this can be confirmed by looking at Theorem \ref{thm:tccopula}. Figure \ref{fig:ca=1b=35less1} shows that if $\alpha_2$ is larger than $\alpha_1$ then in the unit cube the density increases  more sharply for points where $v>u$ for $v\rightarrow u$ than for points where $u>v$ for $u\rightarrow v$. Comparing the rest of the right-hand figures shows that for larger values of  both $\alpha_1$ and $\alpha_2$ the density in the unit cube rises more steeply on both sides of the line $u=v$. These observations are evidence that the increase in $c_{X_1,X_2}$ as the line $u=v$ is approached is affected by both $\alpha_1$ and $\alpha_2$, with steepness increasing on both sides of the line as $\alpha_1$, $\alpha_2$ or both increase.

Additionally, looking at all the plots it is seen that $c_{X_1,X_2}$ increases as we approach the the line $u=v$. Looking at Equation (\ref{sklar}) this shows that $X_1$ and $X_2$ are more likely to take values where $F_{X_1}\approx F_{X_2}$.

Lastly, the figures show that for low $\alpha_1$ and $\alpha_2$ $(\alpha_i\leq2)$ the density clearly differs on the line $u=v$, with events with {\it u} and {\it v} being  low/high being more likely than events with {\it u} and {\it v} close to 0.5. This is not seen when $(\alpha_1,\alpha_2)=(30,35)$ where the density is more evenly spread on the line $u=v$, with all the values in this region having a higher density when compared to the rest of the plane.

\begin{figure}[H]
        \centering
        \begin{subfigure}[b]{0.35\textwidth}
                \centering
                \includegraphics[width=\textwidth]{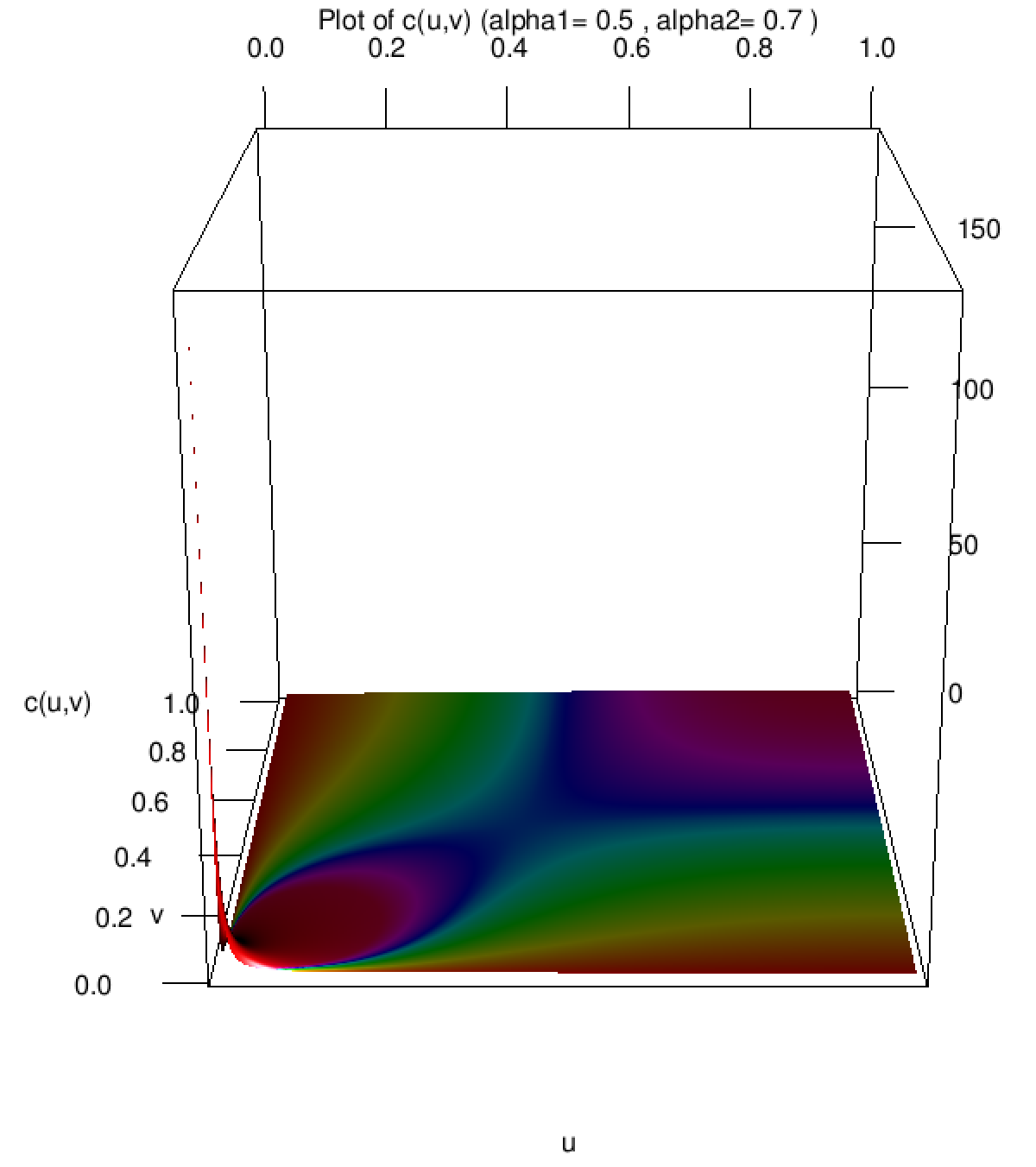}
                \caption{$(\alpha_1,\alpha_2)=(0.5,0.7)$.}
                \label{fig:ca=0.5b=0.7}
        \end{subfigure}%
        ~ 
        \begin{subfigure}[b]{0.35\textwidth}
                \centering
                \includegraphics[width=\textwidth]{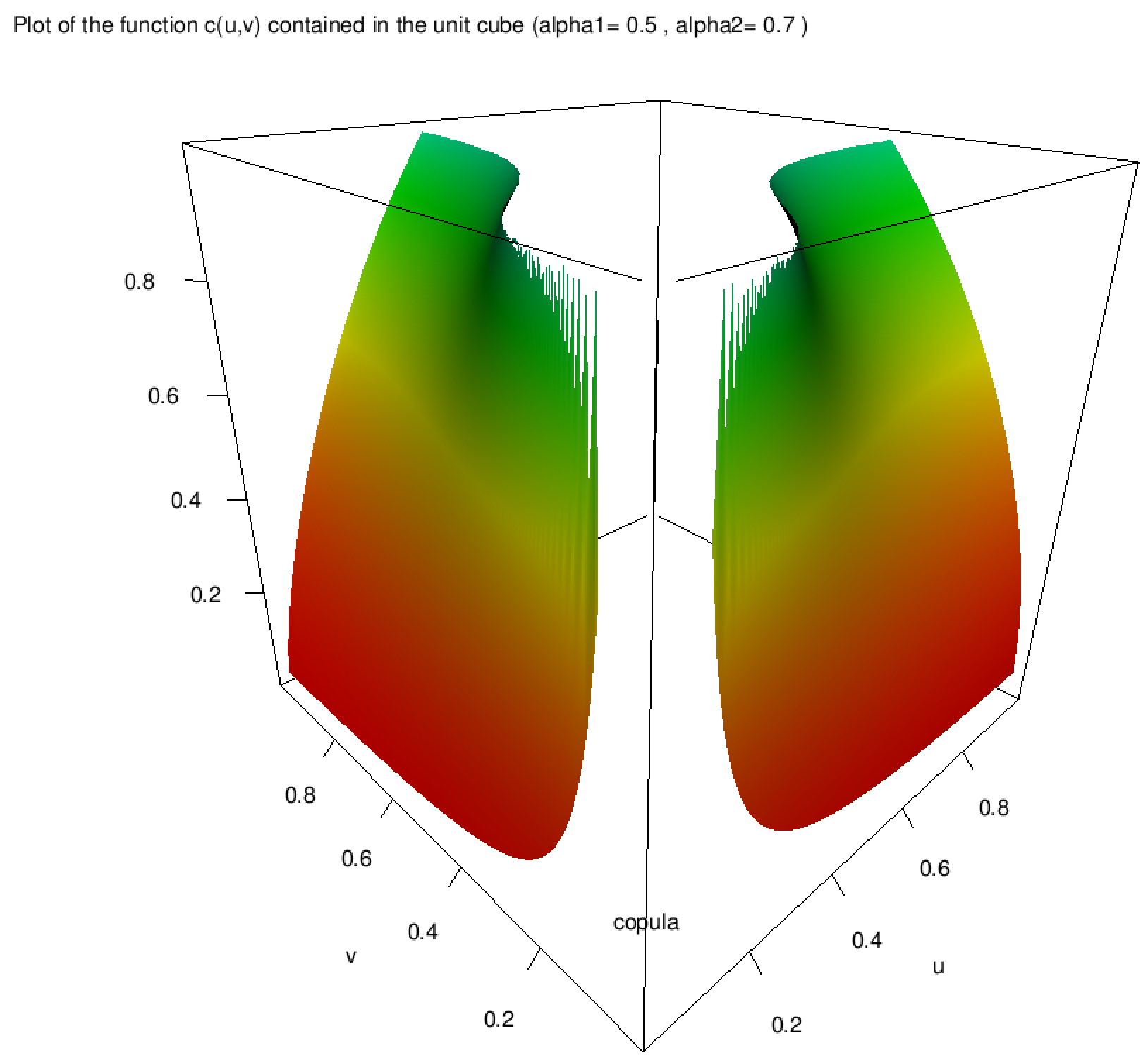}
                \caption{$(\alpha_1,\alpha_2)=(0.5,0.7)$.}
                \label{fig:ca=0.5b=0.7less1}
        \end{subfigure}
        
        \begin{subfigure}[b]{0.35\textwidth}
                \centering
                \includegraphics[width=\textwidth]{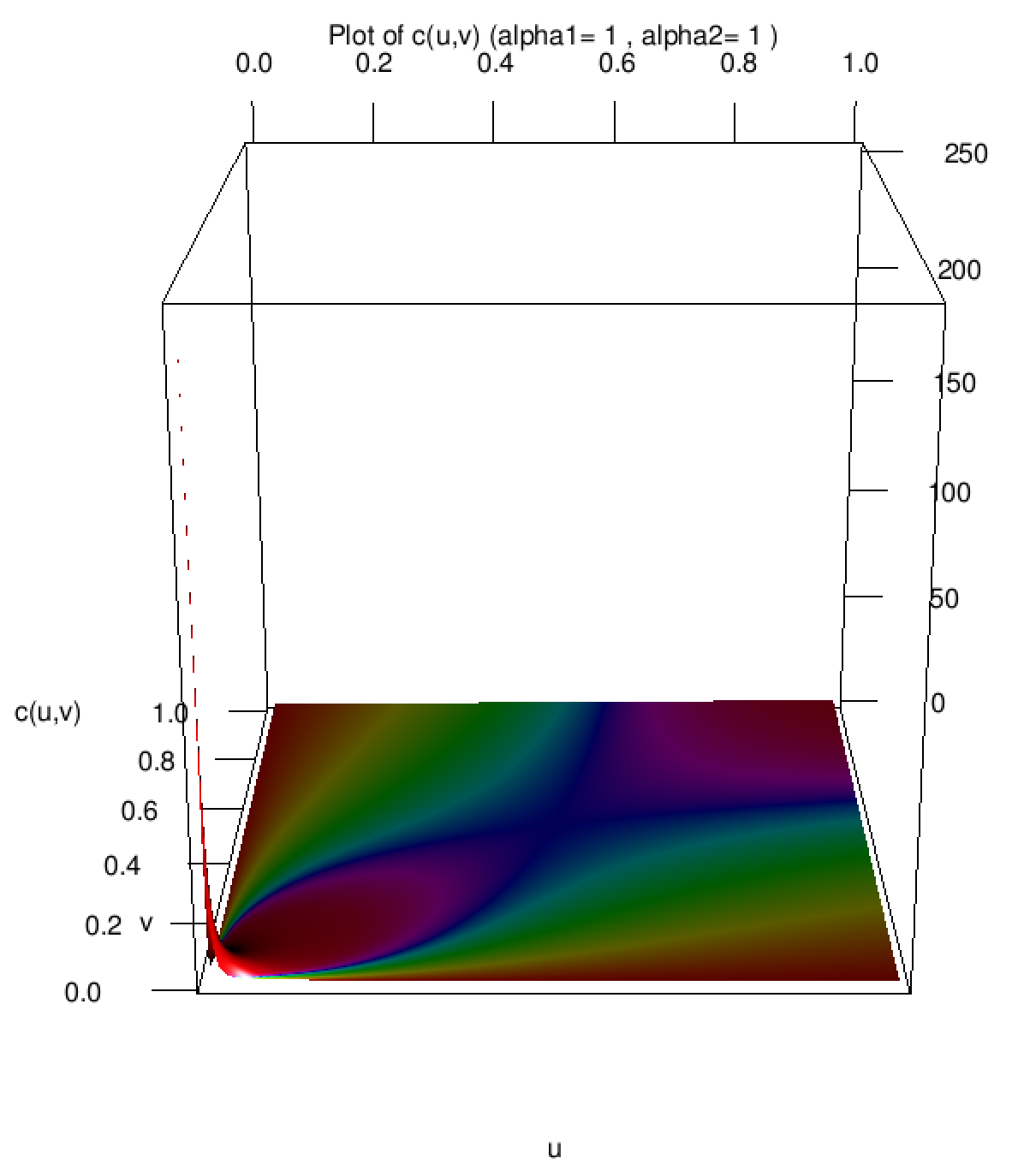}
                \caption{$(\alpha_1,\alpha_2)=(1,1)$.}
                \label{fig:ca=1b=1}
        \end{subfigure}%
        ~ 
        \begin{subfigure}[b]{0.35\textwidth}
                \centering
                \includegraphics[width=\textwidth]{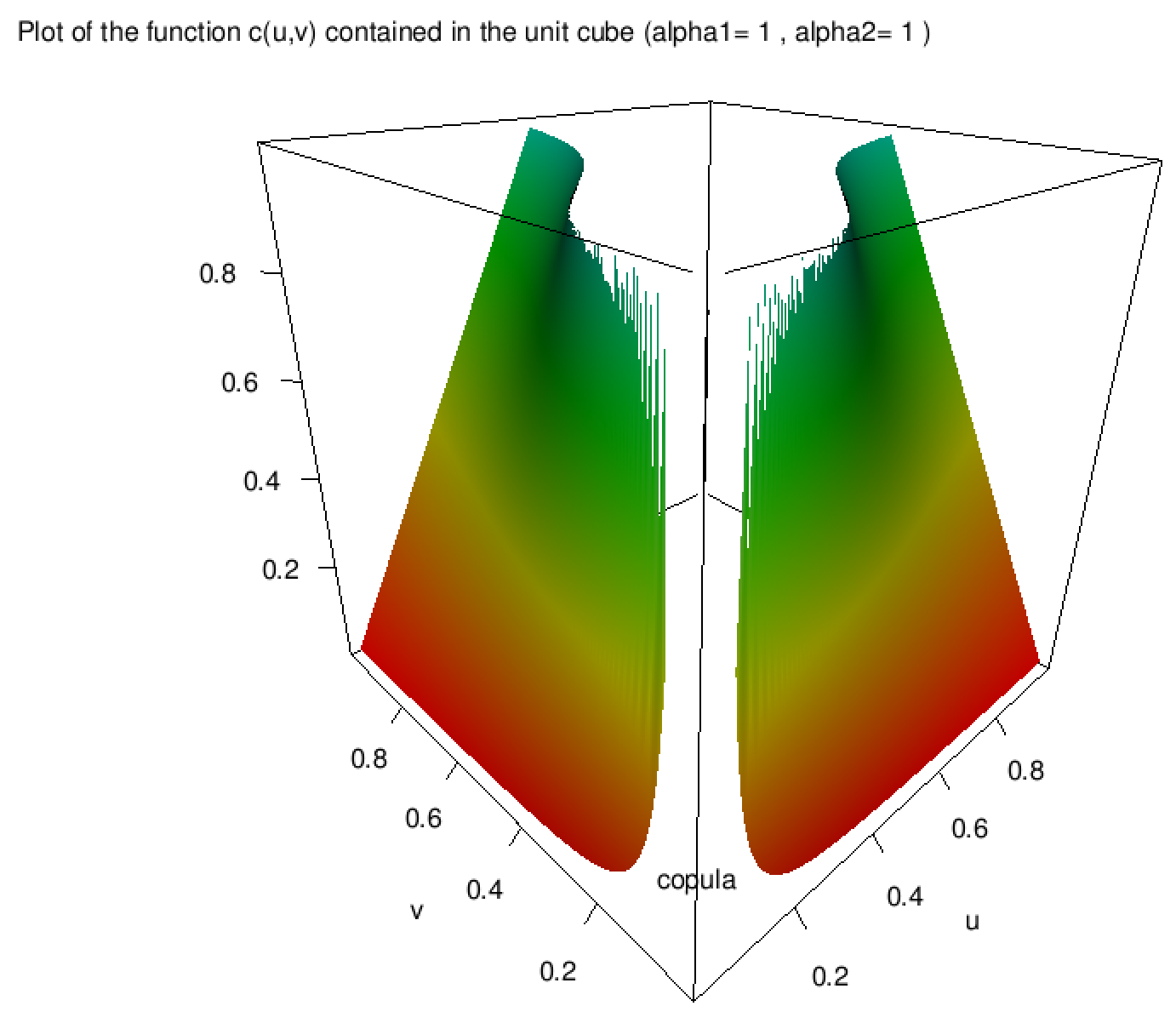}
                \caption{$(\alpha_1,\alpha_2)=(1,1)$.}
                \label{fig:ca=1b=1less1}
        \end{subfigure}
        
        \begin{subfigure}[b]{0.35\textwidth}
                \centering
                \includegraphics[width=\textwidth]{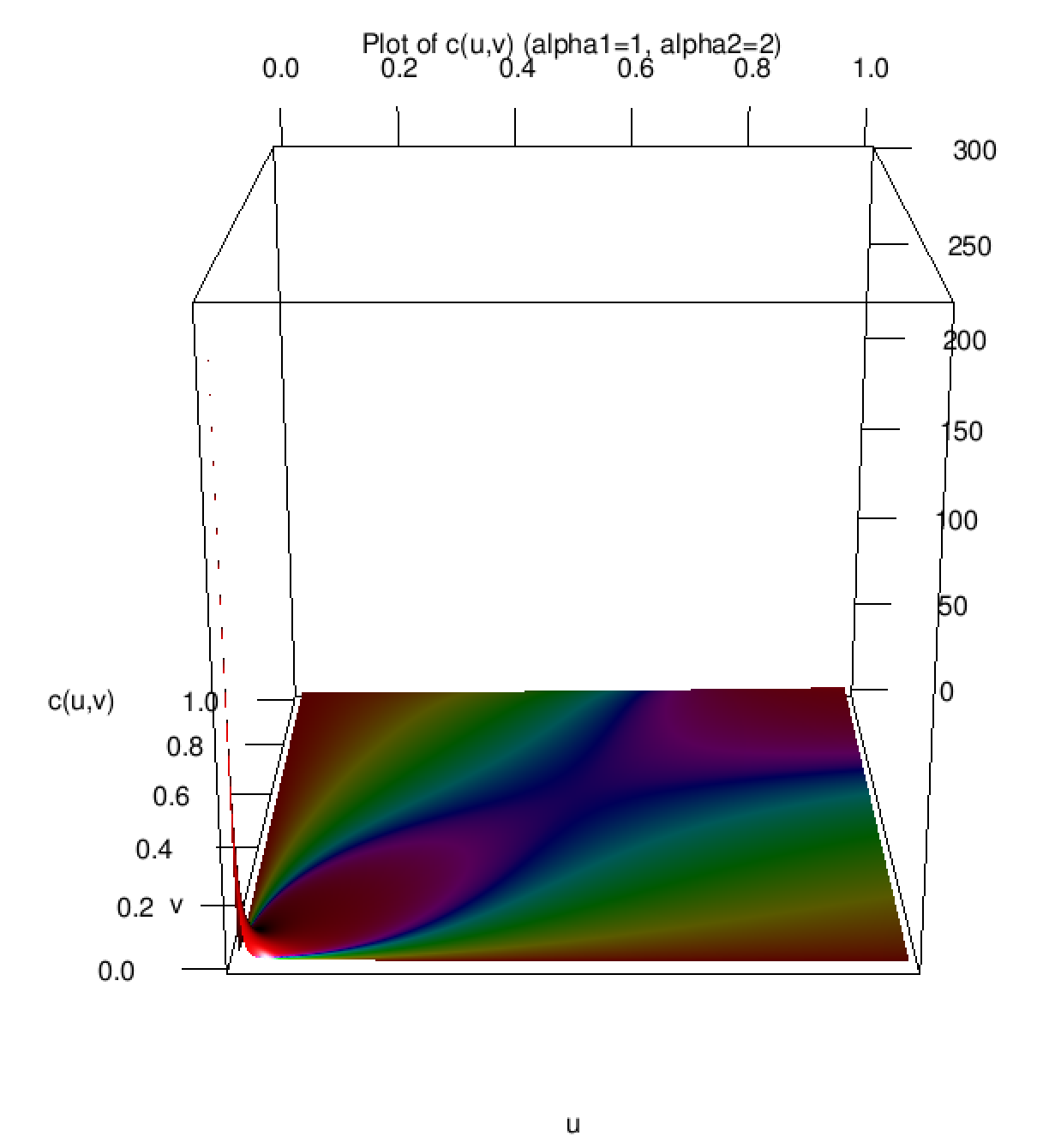}
                \caption{$(\alpha_1,\alpha_2)=(1,2)$.}
                \label{fig:ca=1b=2}
        \end{subfigure}%
        ~ 
        \begin{subfigure}[b]{0.35\textwidth}
                \centering
                \includegraphics[width=\textwidth]{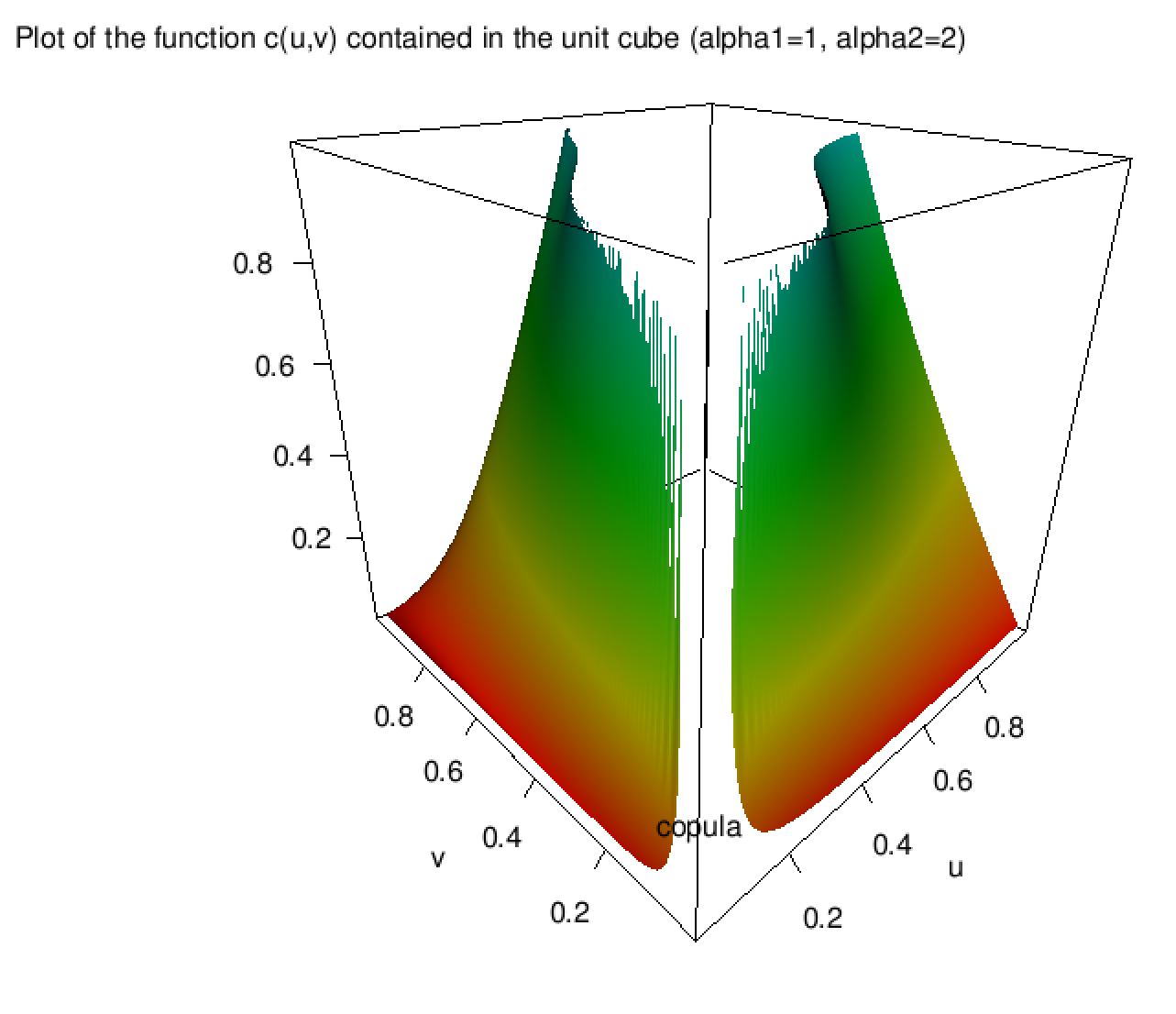}
                \caption{$(\alpha_1,\alpha_2)=(1,2)$.}
                \label{fig:ca=1b=2less1}
        \end{subfigure}
         \caption{Plots of $c_{X_1,X_2}(u,v)$ for different $\alpha_1$ and $\alpha_2$ values. Right plots show the whole $c_{X_1,X_2}(u,v)$ graph, left plots show the part of the $c_{X_1,X_2}(u,v)$ which falls in the unit cube. See Appendix \ref{appendix:colourPalette} for the colour key.}
        \label{fig:twoCompDensity1}
\end{figure}

\begin{figure}[H]
        \centering
        \begin{subfigure}[b]{0.35\textwidth}
                \centering
                \includegraphics[width=\textwidth]{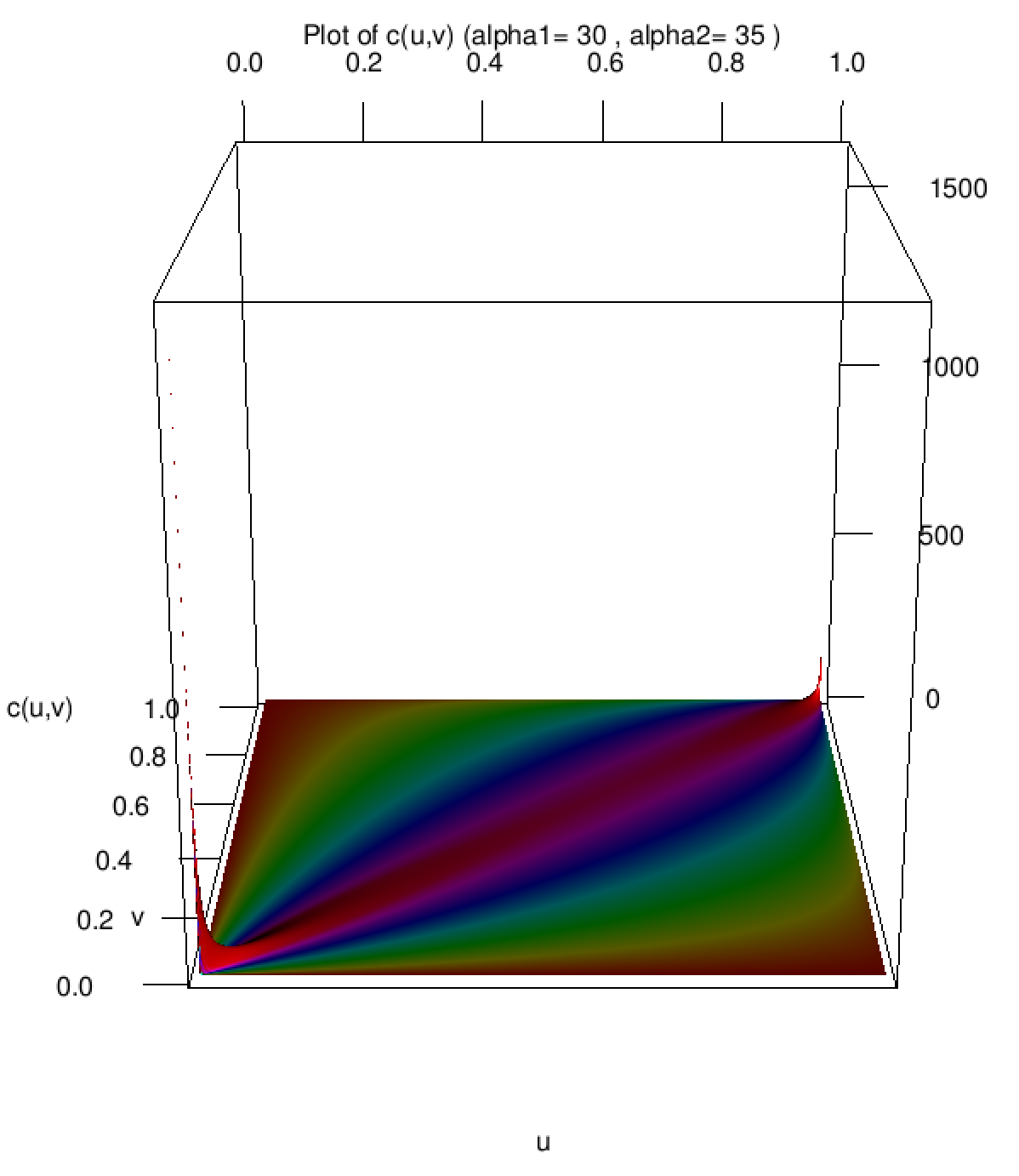}
                \caption{$(\alpha_1,\alpha_2)=(30,35)$.}
                \label{fig:ca=30b=35}
        \end{subfigure}%
        ~ 
        \begin{subfigure}[b]{0.35\textwidth}
                \centering
                \includegraphics[width=\textwidth]{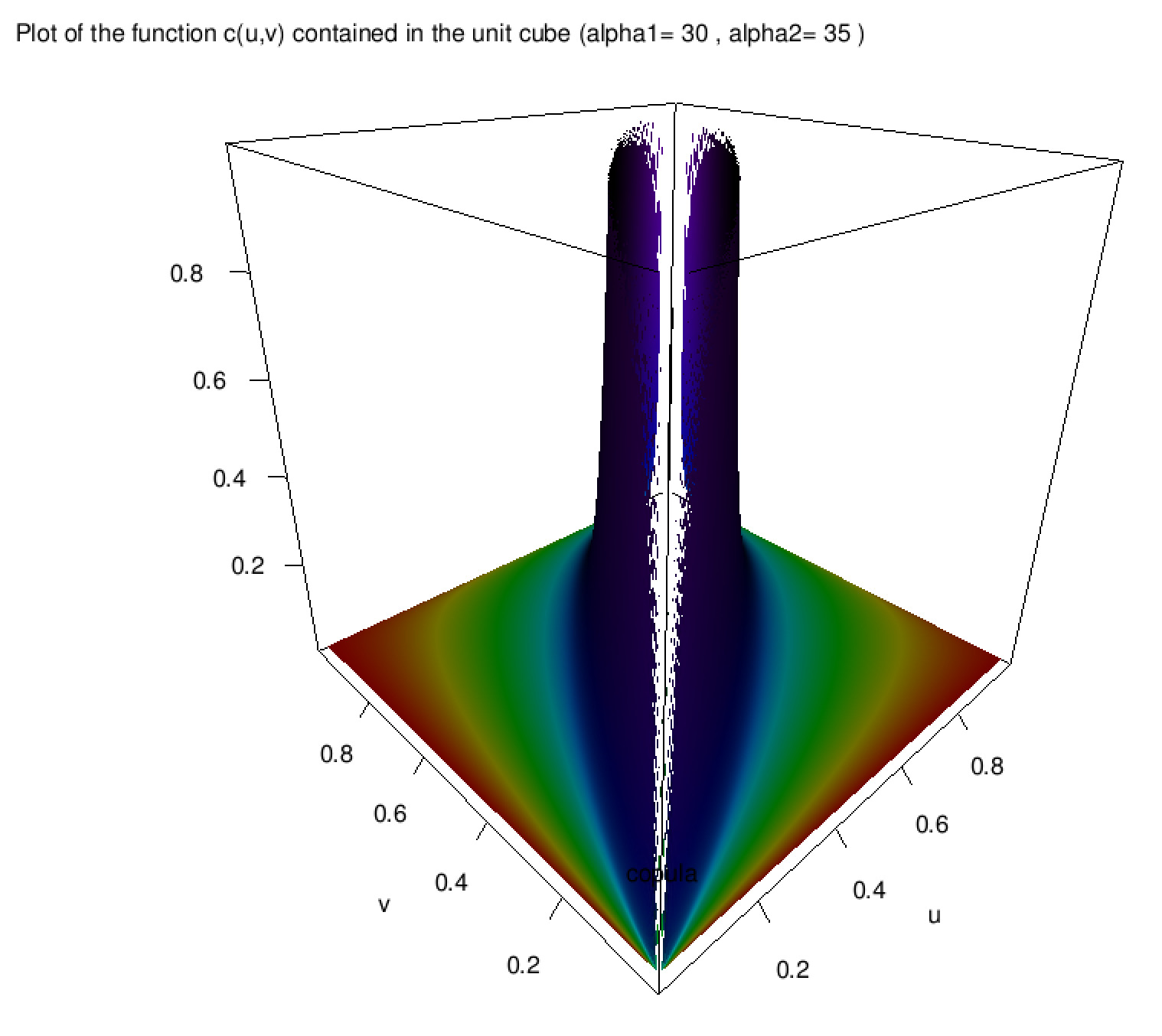}
                \caption{$(\alpha_1,\alpha_2)=(30,35)$.}
                \label{fig:ca=30b=35less1}
        \end{subfigure}
        
        \begin{subfigure}[b]{0.35\textwidth}
                \centering
                \includegraphics[width=\textwidth]{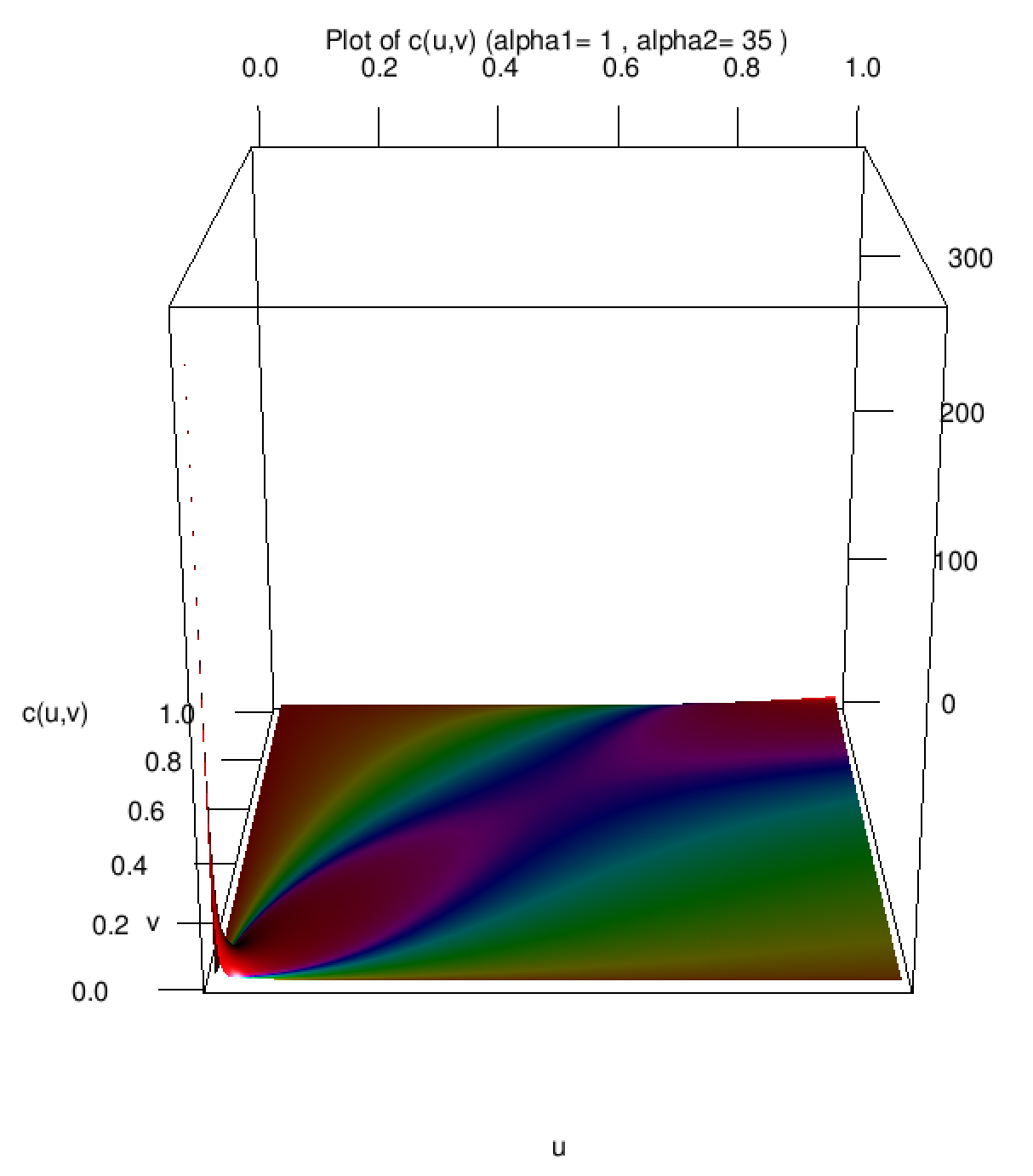}
                \caption{$(\alpha_1,\alpha_2)=(1,35)$.}
                \label{fig:ca=1b=35}
        \end{subfigure}%
        ~ 
        \begin{subfigure}[b]{0.35\textwidth}
                \centering
                \includegraphics[width=\textwidth]{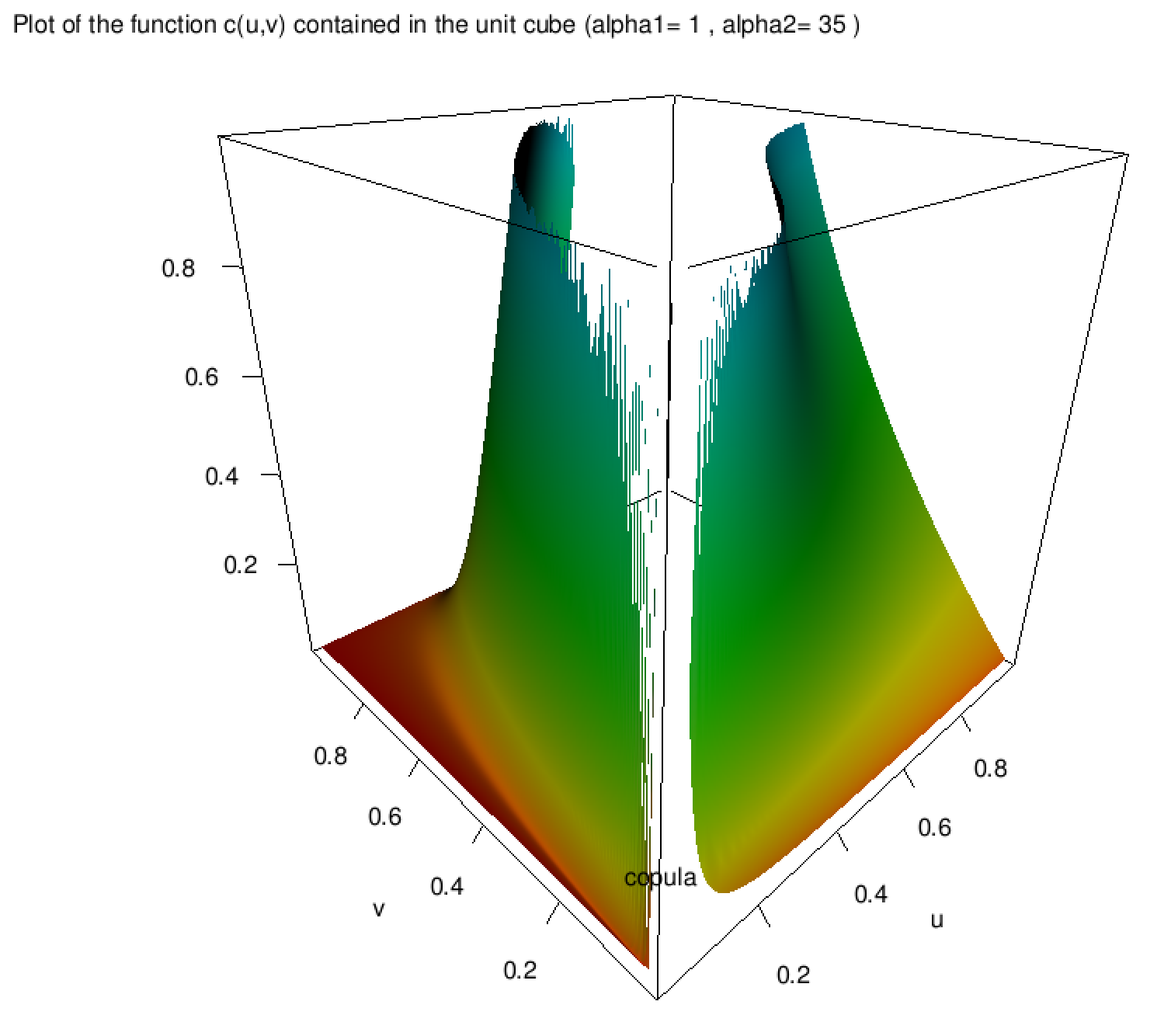}
                \caption{$(\alpha_1,\alpha_2)=(1,35)$.}
                \label{fig:ca=1b=35less1}
        \end{subfigure}        
        \caption{Plots of $c_{X_1,X_2}(u,v)$ for different $\alpha_1$ and $\alpha_2$ values. Right plots show the whole $c_{X_1,X_2}(u,v)$ graph, left plots show the part of the $c_{X_1,X_2}(u,v)$ which falls in the unit cube. See Appendix \ref{appendix:colourPalette} for the colour key.}
        \label{fig:twoCompDensity2}
\end{figure}

\subsection{Upper tail dependency}
Direct verification shows that the upper tail dependence of $C_{X_1,X_2}$ (as defined in Theorem \ref{thm:tccopula})  is
\begin{equation}\label{eq:uppertailtc}
\begin{aligned}
\lambda_U=&\lim_{t\downarrow0}\left(\frac{t^{(\frac{1}{\alpha_1}-1)}}{\alpha_1}\int^\infty_0f_{G_1}(wt^{\frac{1}{\alpha_1}})F_{G_2}(wt^{\frac{1}{\alpha_2}})we^{-w}dw\right.\\
&\left.+\frac{t^{(\frac{1}{\alpha_2}-1)}}{\alpha_2}\int^\infty_0F_{G_1}(wt^{\frac{1}{\alpha_1}})f_{G_2}(wt^{\frac{1}{\alpha_2}})we^{-w}dw\right).
\end{aligned}
\end{equation}
In particular,  if $\alpha_1, \alpha_2<1$ or if $\alpha_1=\alpha_2=1$, we have $\lambda_U=0$.

A simple expression for $\lambda_U$ when $\alpha_1, \alpha_2>1$ was not found, so it is estimated  by plotting the function $\lambda_U(t)$ (Equation (\ref{eq:uppertailtc}) without the limit),  for the given $\alpha_1$ and $\alpha_2$, in a range close to 0 and looking at the behaviour of the graph. These plots will not be mentioned again when the estimated $\lambda_U$ values are given, but they can be found in Appendix \ref{appendix:lambdaplots}.

\subsection{Comparisons to other copulas}
Tthe Two-component copula is not a Gaussian or Archimedean copula as it does not have the symmetry property if $G_1$ and $G_2$ have different $\alpha$ parameters. Further, this copula is not an Extreme-value copula as it does not fit the form given in \ref{def:extremValueC}.

However, while our copula is new to our knowledge, it has a resemblance to the Clayton copula. The Clayton copula is an Archimedean copula with generator $\varphi(t)=\left(\frac{t^{-\theta}-1}{\theta}\right)$, where $\theta>0$ to enforce $\varphi(0)=\infty$, , see \cite{schmidt}. 

\begin{proposition}
If $U$ and $V$ are generated by
$$
U=\left(1+\frac{W_1}{S}\right)^{\frac{1}{\theta}},\quad V=\left(1+\frac{W_2}{S}\right)^{\frac{1}{\theta}},
$$
where $W_1, W_2$ and $S$ are independent with $W_i\sim Exp(1)$ and $S\sim\Gamma(\frac{1}{\theta},1)$, then U and V have a Clayton copula with parameter $\theta$.
\end{proposition}
\begin{proof}
\cite{schmidt} gives that if $R_1, R_2$ and $S$ are independent with $R_i\sim U(0,1)$ and $S\sim\Gamma(\frac{1}{\theta},1)$, then 
$$
U=\left(1-\frac{ln(R_1)}{S}\right)^{\frac{1}{\theta}},\quad V=\left(1-\frac{ln(R_2)}{S}\right)^{\frac{1}{\theta}},
$$
 have a Clayton copula with parameter $\theta$.
Now, 
\begin{equation*}
\begin{aligned}
P(-ln(R_i)\leq r) &=P(R_i\geq e^{-r})=\left\{ 
  \begin{array}{l l}
    0 & \quad \text{if $r<0$}\\
    1-e^{-r} & \quad \text{if $0\leq r \leq \infty$}
  \end{array} \right.
\end{aligned}
\end{equation*}
Hence $(-ln(R_1))$ and $(-ln(R_1))$ are i.i.d. Exp(1) as required.
\end{proof}

\begin{proposition}
If U and V are given by 
$$
U=\left(1+\frac{W}{G_1}\right)^{\alpha_1},\quad V=\left(1+\frac{W}{G_2}\right)^{\alpha_2},
$$
where $W, G_1$ and $G_2$ are independent with $W\sim Exp(1)$ and $G_i\sim\Gamma(\alpha_i,1)$, then U and V have the {\it Two-component model copula} which is given in Theorem \ref{thm:tccopula}, with parameters $\alpha_1$ and $\alpha_2$.
\end{proposition}
\begin{proof}
Define $X_1$ and $X_2$ as in equations (\ref{eq:X1}) and (\ref{eq:X2}), then $X_1$ and $X_2$ have a Two-component copula with parameters  $\alpha_1$ and $\alpha_2$. Define 
$$
U_1=1-\left(1+\frac{X_1}{\sigma_1}\right)^{-\alpha_1},\quad V_1=1-\left(1+\frac{X_2}{\sigma_2}\right)^{-\alpha_2}.
$$
Then in distribution, 
$$
U_1=1-\left(1+\frac{W}{G_1}\right)^{-\alpha_1},\quad V_1=1-\left(1+\frac{W}{G_2}\right)^{-\alpha_2}, 
$$
where $W$, $G_1$ and $G_2$ are independent with $W\sim Exp(1)$ and $G_i\sim\Gamma(\alpha_i,1)$. Further, $U_1$ and $V_1$ take values on [0,1) as $X_1, X_2>0$. Now, as $U_1$ and $V_1$ are strictly increasing transformations of $X_1$ and $X_2$, $U_1$ and $V_1$ have a Two-component copula with parameters $\alpha_1$ and $\alpha_2$ by \eqref{thm:copTransformations2}. Define $U=1/(1-U_1)$ and  $V=1/(1-V_1)$, then again, by \eqref{thm:copTransformations2}, $U$ and $V$ are strictly increasing transformations of $U_1$, $V_1$, so $U$ and $V$ also have a Two-component copula with parameters $\alpha_1$ and $\alpha_2$.
\end{proof}
Hence in the Two-component copula the Exponential distribution is fixed and the Gamma distribution and parameter $\alpha$ varies between $U$ and $V$, while in the Clayton copula the Exponential distribution varies and the Gamma distribution and parameter $\theta$ are fixed between $U$ and $V$.

\subsection{Goodness-of-fit testing on simulated data}
Tosee how well our goodness-of-fit tests preform on data simulated from the Two-component model, we  first simulate 1000 i.i.d. observations of $(X_1,X_2)$ from (\ref{eq:X1}) and (\ref{eq:X2}). We assume that loss function 1 has the larger propensity for loss, and normalise both loss function's scale parameters according to loss functions 1's scale parameter; hence we choose  $\sigma_1 =1$.  For simplicity we assume loss function 2's propensity for loss is in the ratio of 9:10 when compared to loss function 1; hence $\sigma_2 = 0.9$. As copulas are invariant under monotone increasing transformations of random variables, this simplification does not affect any of our copula fits or GoF tests. Moreover  from our assumptions we have $\xi_1=\frac{1}{\alpha_1}$ lies in $(0,1]$. As we do not have any other assumptions regarding this parameter we pick its value uniformly in the interval $(0,1]$; for the same reason we pick  $\xi_2=\frac{1}{\alpha_2}$ uniformly in the interval $(0,1]$ as well.

 Then we fit the best Gaussian, Gumbel and Two-component copulas to the data and perform a goodness-of-fit test in each case, as well as the more general goodness-of-fit test, which tests whether the data comes from an Extreme-value copula, as discussed in Section \ref{chap:background}. For the Two-component copula the parameters are estimated by the maximum likelihood estimated shape parameters of the marginal GPDs of $X_1$ and $X_2$. We then apply the tests from Section \ref{chap:background}. 
 As we carry out $m(\leq4)$ GoF tests we apply the generalised Benjamini-Hochberg procedure (Theorem $1.3$ \cite{pap:BH}): for a test at level $\beta$ we reject the $ith$ hypothesis ($H_0^i$) if $p_i<(\beta/\sum_{j=1}^m\frac{1}{j})$.

Table \ref{tab:SimulatedResults} shows that the Gaussian, Gumbel and Extreme-value copula tests provided statistically significant $p$-values $(0.05/\sum_{j=1}^4\frac{1}{j})$ at the 5\% significance level.   
{
\begin{table} [h!]
  \begin{center}
      \begin{tabular}{ l | l } 
       \multicolumn{2}{c} {\bf Fitted model parameters (3 s.f.)} \\ \hline
      Gaussian $(R_{12})$ & 0.645   \\ 
      Gumbel $(\theta)$ & 1.81  \\ 
      Two-component $(\alpha_1,\alpha_2)$ & (3.05,1.18) \\ 
       \multicolumn{2}{c} {\bf Estimated $\lambda_U$ (3 s.f.)} \\ \hline
      Gaussian & 0 \\ 
      Gumbel & 0.533 \\
      Two-component & 0 \\ 
       \multicolumn{2}{c} {\bf $P$-values for GoF (3 s.f.)} \\ \hline
      Gaussian & 0 \\ 
      Gumbel & 0  \\ 
      Two-component ($p$-value, $\#$valid iterations) & (0.689,1000) \\ 
      Extreme-value copula & $1.12e^{-31}$  \\ 
      \end{tabular}
      \caption[Caption for LOF]{Results table for fitting copulas to data simulated from the  Two-component model with $\alpha_1 =3.387732  $ and 
      $\alpha_2= 1.181292$,  $\sigma_1=1, \sigma_2=0.9$. The $p$-values are given before the generalised Benjamini-Hochberg procedure.\label{tab:SimulatedResults}}
  \end{center}
\end{table}
}

There is evidence to reject the hypotheses that the dataset has a Gaussian, Gumbel or Extreme-value copula. Since the data did come from a Two-component copula, this result suggests that the Two-component copula is very different to both the Gaussian and Gumbel copulas and illustrates that  is not an Extreme-value copula. This result for the Gaussian and Gumbel copula is emphasised in Figure \ref{fig:TcStatisticsHist.png} in  Appendix A, which shows that the observed test statistics for the Gaussian and Gumbel tests fell in the extreme tail of the bootstrap simulated distribution of the test statistic under the respected null hypotheses (see Step 5 in Appendix \ref{appendix:bootstrap} for the simulation method). The Two-component copula goodness-of-fit test did not provide significant results and the estimated Two-component copula parameters were close to the real values. This is reassuring as we know the dataset is indeed generated from a Two-component copula.

Figure \ref{fig:Tcfittedcopulas} shows that the Two-component copula was the best fitting copula of the simulated Two-component model data, as it has the most overlap with the empirical copula while the other copulas show a bias (over-approximate the copula for lower values of $u$ and $v$, while under-approximate the copula for higher values of $u$ and $v$).

\begin{figure}[H]
        \centering
        \begin{subfigure}[b]{0.35\textwidth}
                \centering
                \includegraphics[width=\textwidth]{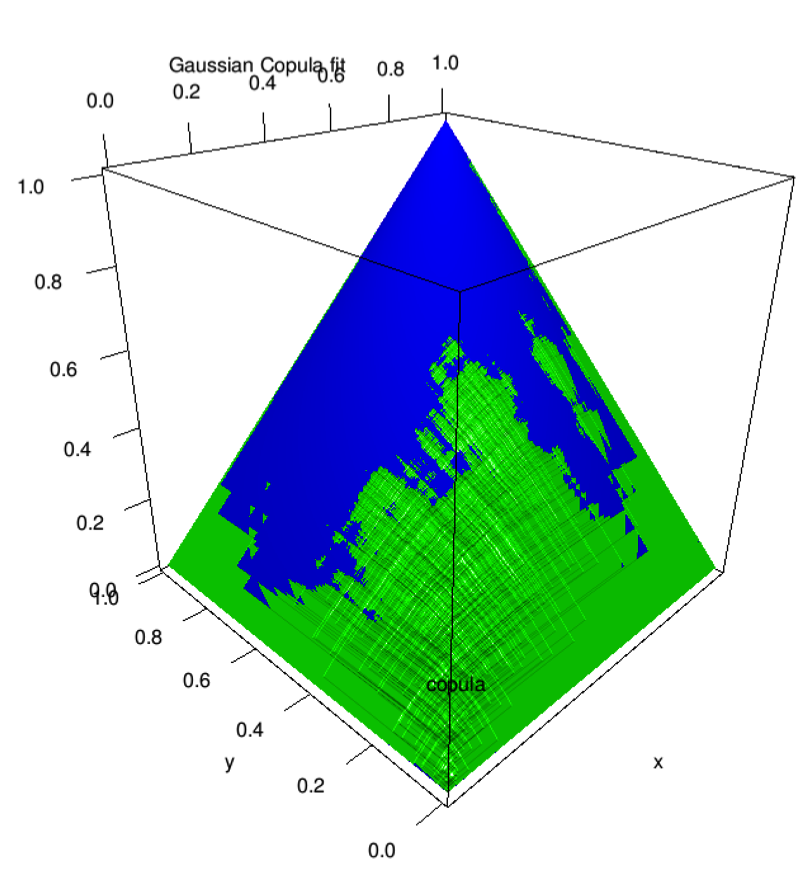}
                \caption{Gaussian Copula.}
                \label{fig:TCGauUp.png}
        \end{subfigure}%
        ~ 
        \begin{subfigure}[b]{0.35\textwidth}
                \centering
                \includegraphics[width=\textwidth]{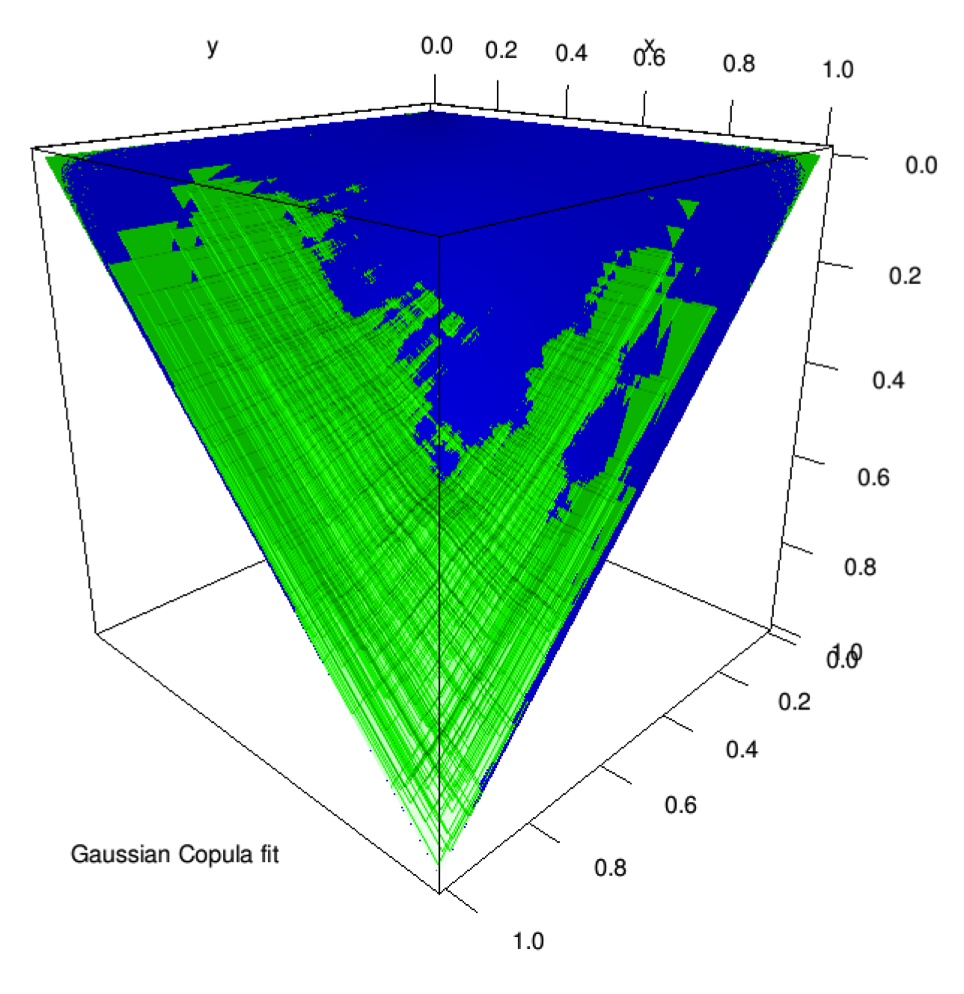}
                \caption{Gaussian Copula.}
                \label{fig:TCGauDown.png}
        \end{subfigure}
        
        \begin{subfigure}[b]{0.35\textwidth}
                \centering
                \includegraphics[width=\textwidth]{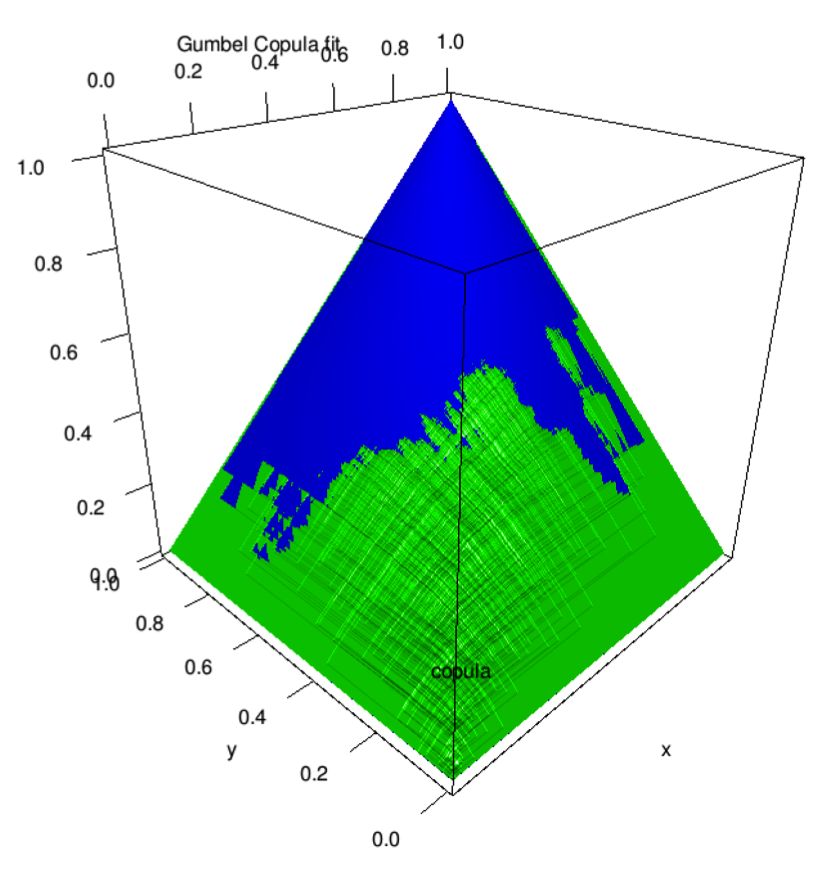}
                \caption{Gumbel Copula.}
                \label{fig:TCGumUp.png}
        \end{subfigure}%
        ~ 
        \begin{subfigure}[b]{0.35\textwidth}
                \centering
                \includegraphics[width=\textwidth]{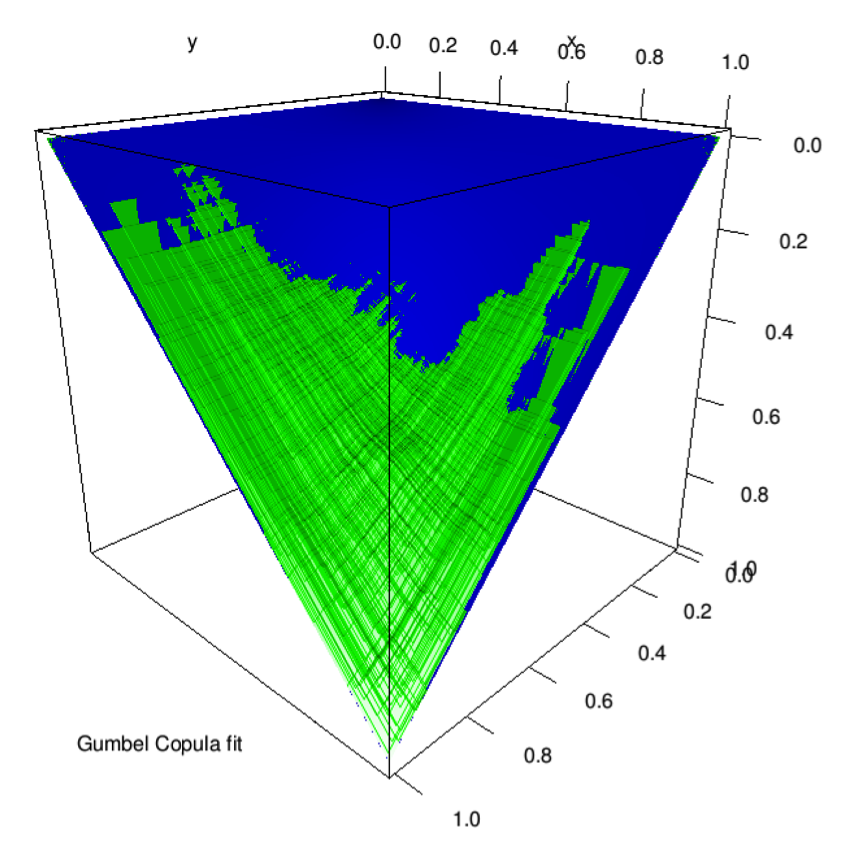}
                \caption{Gumbel Copula.}
                \label{fig:TCGumDown.png}
        \end{subfigure}
        
        \begin{subfigure}[b]{0.35\textwidth}
                \centering
                \includegraphics[width=\textwidth]{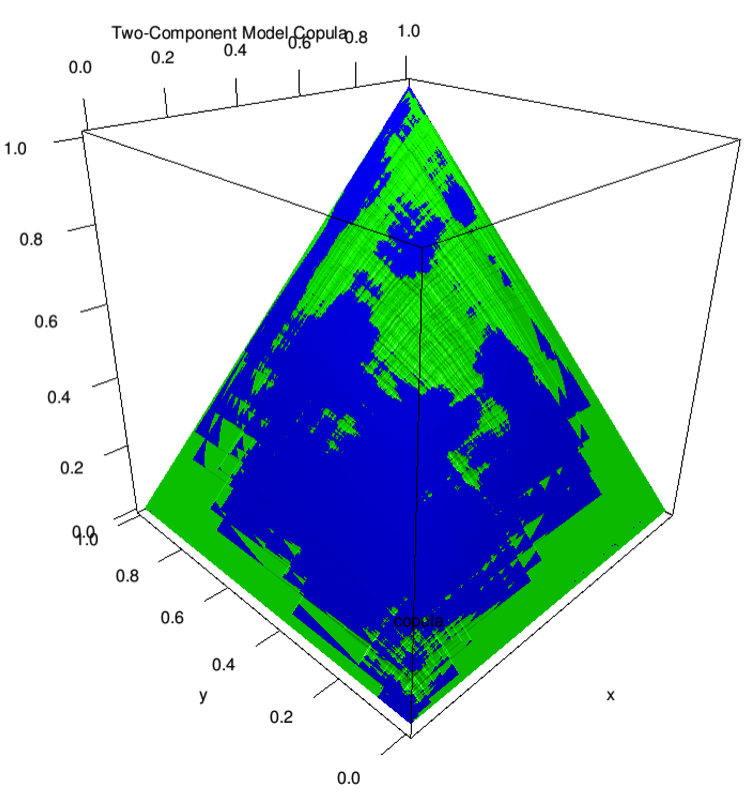}
                \caption{Two-component Copula.}
                \label{fig:TCTcUp.png}
        \end{subfigure}%
        ~ 
        \begin{subfigure}[b]{0.35\textwidth}
                \centering
                \includegraphics[width=\textwidth]{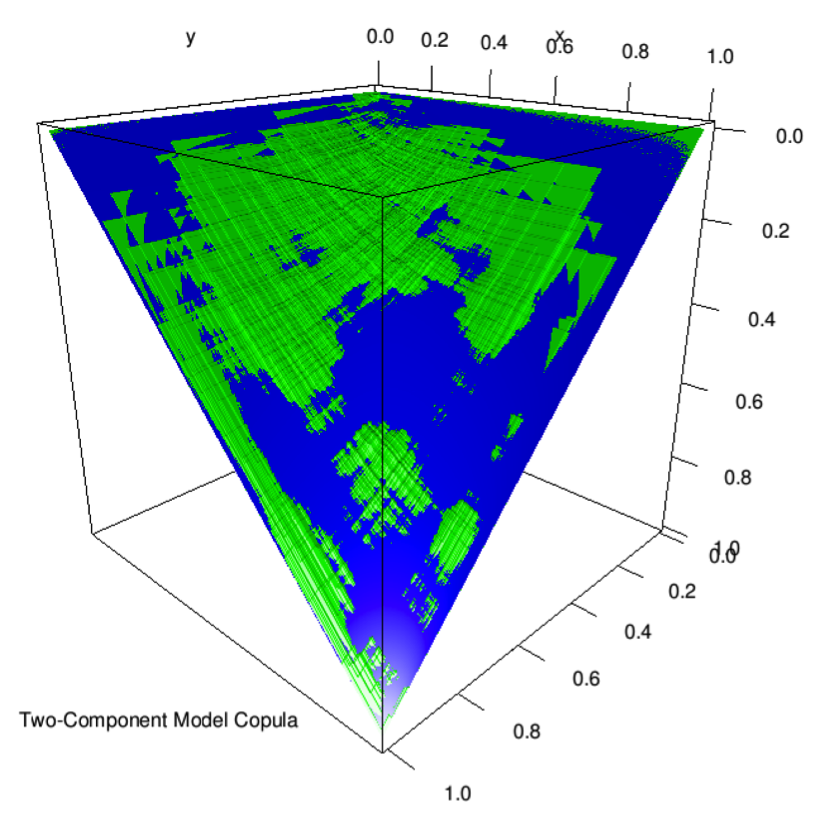}
                \caption{Two-component Copula.}
                \label{fig:TCTcDown.png}
        \end{subfigure}        
        \caption{3-Dimensional plots comparing the Empirical copula (green) with the best fitting parametric copulas (blue) for the simulated data from the Two-component model.The right plots show a view of the graph where the z-axis is increasing upwards,  the left plots show a view of the graph where the z-axis is increasing downwards.}
        \label{fig:Tcfittedcopulas}
\end{figure}

Thus, the Two-component copula is unlike some of the most commonly used copulas in the insurance field to-date. Consequently, it fills a gap in the literature on copulas for insurance applications.

\bibliographystyle{alea3}

\bibliography{refs}

\newpage 
\appendix  

\section{Bootstrap method and test results}\label{appendix:bootstrap}
First  we describe the bootstrap method used to obtain a $p$-value in the copula goodness-of-fit test based on the empirical copula. For reference of this method, see section 3.10 of \cite{Berg}.
\begin{enumerate}
	\item Using Equation (\ref{eq:transformedsample}), generate the transformed sample (${\bf u_1,...,u_n}$) from the sample data (${\bf x_1,...,x_n}$). 
	\item Estimate the parameters of the parametric copula, $\hat{\theta}$, from the transformed sample, and ensure they satisfy any requirements needed to make the parametric copula valid. If not, no valid parametric copula fits the data and the test fails.
	\item Compute the empirical copula, $C_n$, using Equation (\ref{equ:empCopula}).
	\item \label{step:boot2} Estimate the test statistic $\hat{\rho}_{CvM}$ by plugging  $C_n$,  $C_{\hat{\theta}}$ and (${\bf u_1,...,u_n}$) into Equation (\ref{equ:CvM}). 
	\item For some large integer K, repeat the following steps for every $k\in\{1,...,K\}$ (parametric bootstrap):
	\begin{enumerate}[(i)]
		\item Generate a random sample $({\bf x_{1,k}^0,...,x_{n,k}^0})$ from the null hypothesis copula $C_{\hat{\theta}}$ and using (\ref{eq:transformedsample}) calculate the associated transformed sample (${\bf u_{1,k}^0,...,u_{n,k}^0}$).
		\item  Estimate the parameters of the parametric copula, $\hat{\theta}^0$, from the transformed sample (${\bf u_{1,k}^0,...,u_{n,k}^0}$), and ensure they satisfy any requirements needed to make the parametric copula valid. If not pass over this iteration.
		\item\label{step:Boot} Estimate the bootstrap test statistic $\hat{\rho}_{CvM,k}^0$ by plugging  $C_k^0$, $C_{\hat{\theta}^0}$ and (${\bf u_{1,k}^0,...,u_{n,k}^0}$) into Equation (\ref{equ:CvM}),  where $C_k^0$ is the empirical copula of the sample (${\bf u_{1,k}^0,...,u_{n,k}^0}$) using Equation (\ref{equ:empCopula}).
	\end{enumerate}
	\item\label{step:pvalue} Approximate the  $p$-value of the test by $\hat{p}=\frac{1}{V+1}\sum_{k:Valid}{\bf1}_{\hat{\rho}_{CvM,k}^0\geq\hat{\rho}_{CvM}}$, where V is the number of valid iterations and the sum only goes over the valid iterations.
\end{enumerate}

In this study, K is chosen to be 1000.\\

\noindent {\it Note.} Steps \ref{step:boot2} and \ref{step:Boot} only work as there is an analytical expression $C_\theta$ for each of the copulas. If this was not the case then we would carry out the bootstrap method explained in \cite{Berg} to estimate $\hat{\rho}_{CvM}$ and $\hat{\rho}_{CvM,k}^0$ respectively.\\

Next we give more details on the results for the comparison with the Gaussian, Gumbel and extreme-value copula. 
Figure \ref{fig:TcStatisticsHist.png} shows that the observed test statistics for the Gaussian and Gumbel tests fell in the extreme tail of the bootstrap simulated distribution of the test statistic under the respected null hypotheses. It is re-assuring that the simulated values are plausible for a two-component copula, as that is how they were generated. 

\begin{figure}[h!]
\begin{center}
\includegraphics[width=10cm]{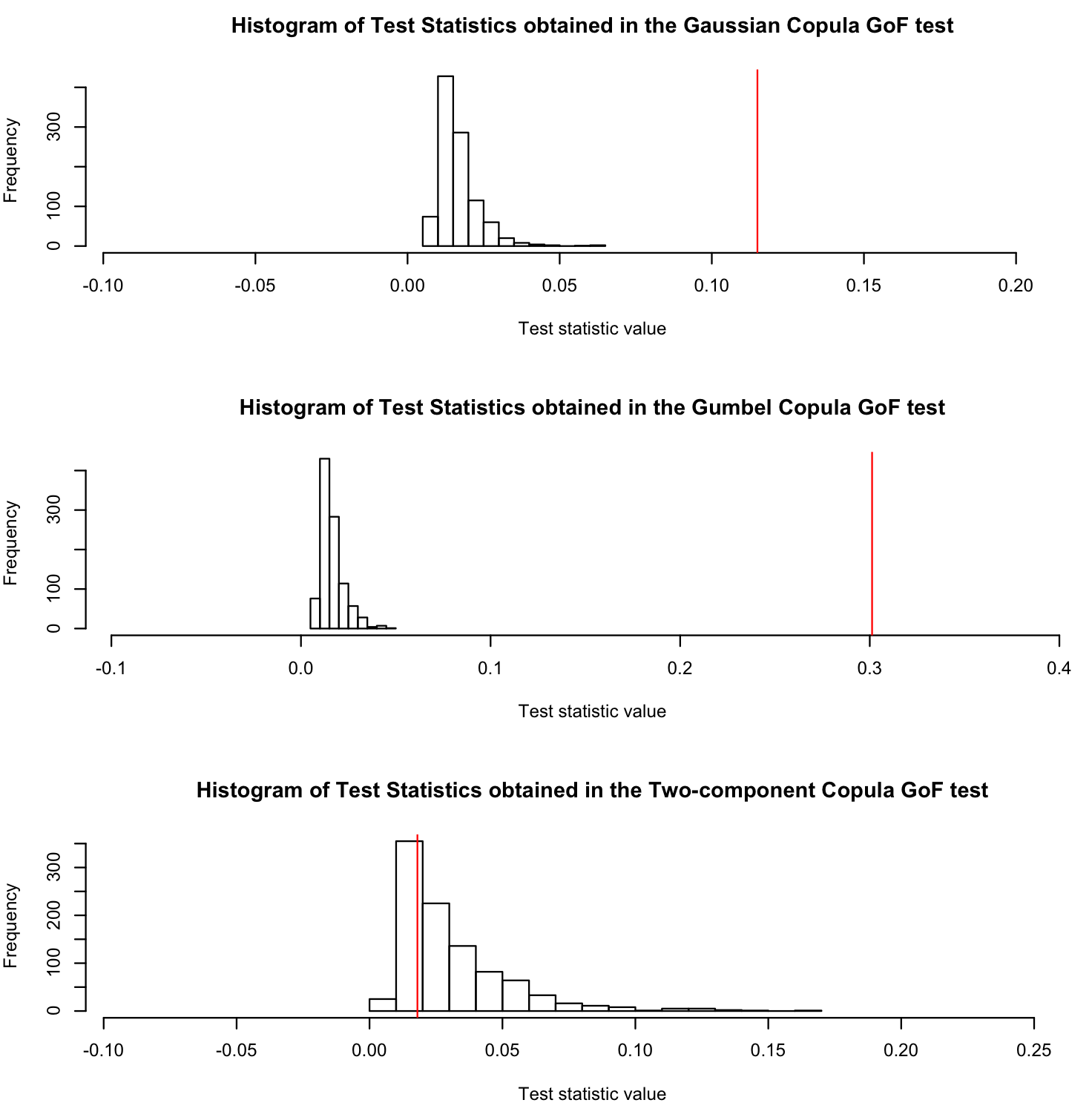}
\caption[Caption for LOF]{A histogram showing the spread of the test statistics in the GoF test for each copula. The red line is the observed test statistic from the simulated data. \label{fig:TcStatisticsHist.png}}
\end{center}
\end{figure}
\section{Plots used to estimate $\lambda_U$ for the Two-component copula} 
This graph was used in Table \ref{tab:SimulatedResults} to estimate the upper tail dependence of a Two-component copula with parameters $\alpha_1=3.387732$ and $\alpha_2=1.181292$. The graph was constructed by plotting equation \ref{eq:uppertailtc} (without the limit) for small $t$. If the graph looked to be convergent close to zero, then the upper-tail dependence of the copula was estimated by picking this convergent value. In this case a value of 0 was picked as it can be seen that the curve monotonically decreases towards 0 as $t$ decreases, and for small $t$ the curve is within 0.001 units of 0.

\label{appendix:lambdaplots}
\begin{figure}[ht]
        \centering
        \begin{subfigure}[b]{0.5\textwidth}
                \centering
                \includegraphics[width=\textwidth]{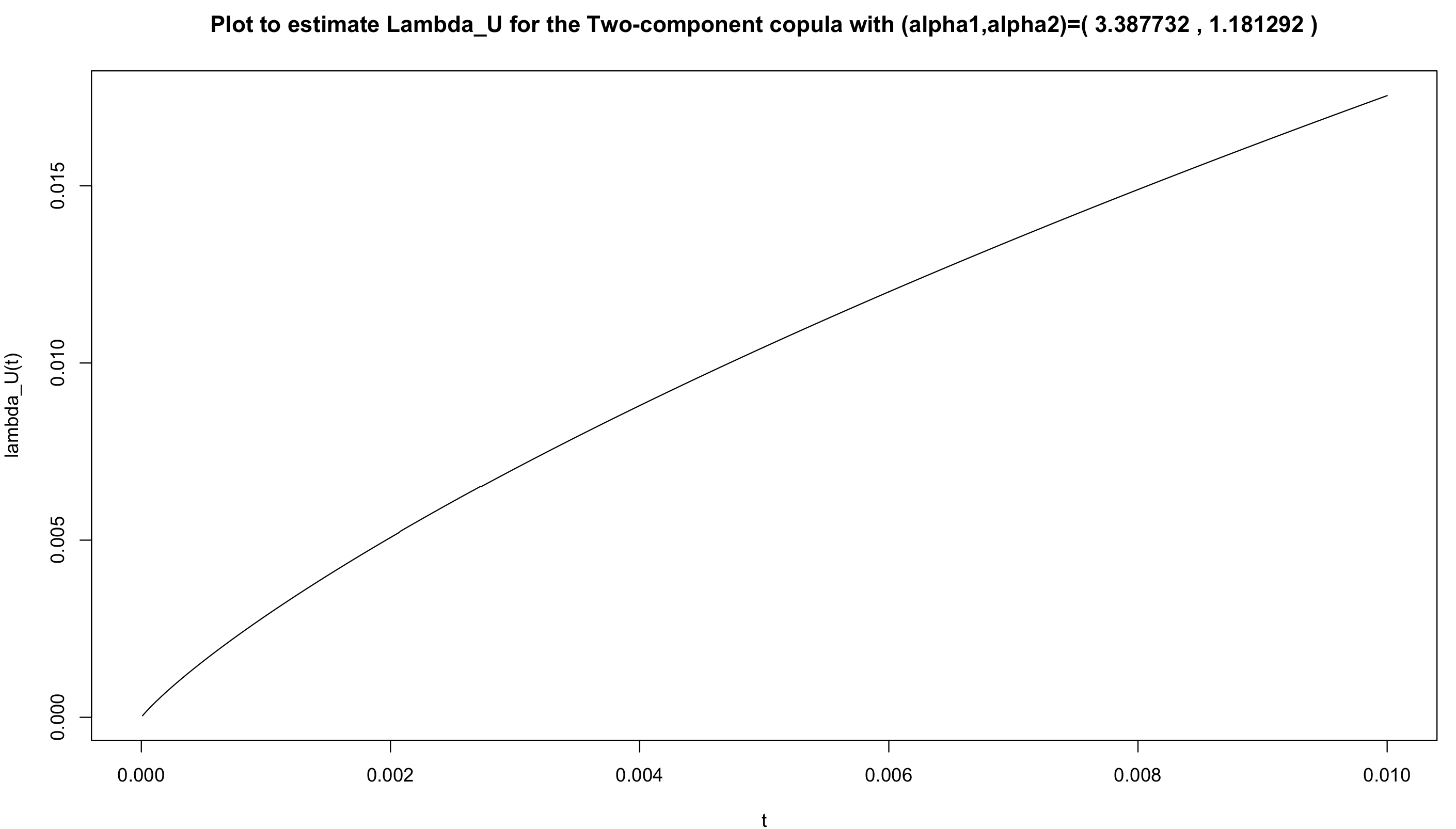}
                \caption{$(\alpha_1,\alpha_2)=(3.387732,1.181292)$.}
                \label{fig:Tcdata.jpg}
        \end{subfigure}%
        \caption{Plot used to estimate $\lambda_U$ for the Two-component copula.}
        \label{fig:GoFdata3.jpg}
\end{figure}

\section{Colour Palette}\label{appendix:colourPalette}
This colour palette was used to shade the plots in Figures \ref{fig:twoCompDensity1} and \ref{fig:twoCompDensity2} to allow the shape of the graph to be seen more clearly. Areas of the graph with larger z-values were shaded using higher ranking colours. The colour palette is the standard rainbow colour palette used in R.

\begin{figure}[htb]
	\begin{center}
		\label{fig:colourPalette}
		\includegraphics[width=\textwidth]{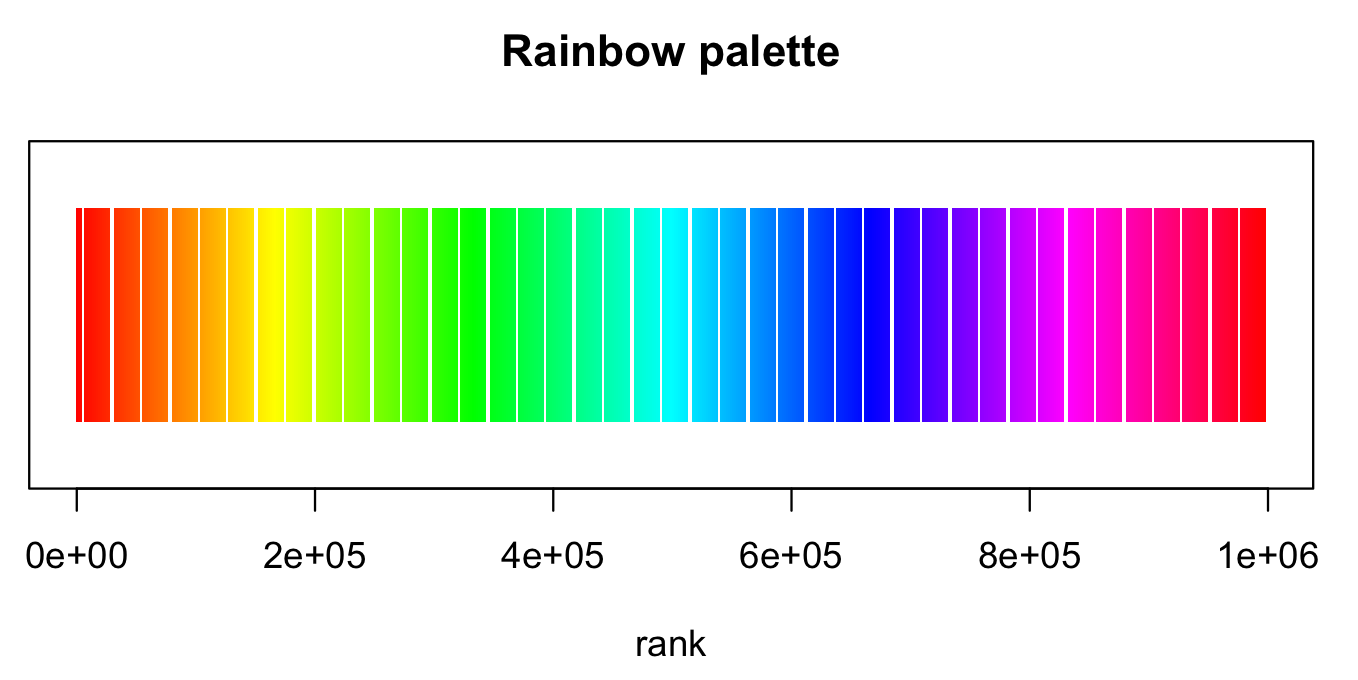}
		\caption[Caption for LOF]{Colour Palette.}
	\end{center}
\end{figure}

\end{document}